\documentclass[preprint, 12pt]{amsart}
\usepackage{amsmath,amssymb,enumerate,comment,amsthm,mathrsfs,bbm}
\usepackage{etoolbox}
\numberwithin{equation}{section}

\newcommand{\IB}{{\mathbb B}}
\newcommand{\IC}{{\mathbb C}}

\newcommand{\IN}{{\mathbb N}}
\newcommand{\IZ}{{\mathbb Z}}

\newcommand{\IK}{{\mathbb K}}
\newcommand{\IT}{{\mathbb T}}

\newcommand{\cA}{{\mathcal A}}
\newcommand{\cB}{{\mathcal B}}
\newcommand{\cE}{{\mathcal E}}

\newcommand{\cH}{{\mathcal H}}
\newcommand{\cK}{{\mathcal K}}
\newcommand{\cC}{{\mathcal C}}

\newcommand{\cL}{{\mathcal L}}
\newcommand{\cM}{{\mathcal M}}
\newcommand{\cN}{{\mathcal N}}
\newcommand{\cO}{{\mathcal O}}
\newcommand{\cS}{{\mathcal S}}

\newcommand{\fA}{{\mathfrak A}}

\newcommand{\fs}{{\mathfrak s}}
\newcommand{\acts}{\curvearrowright}

\newcommand{\id}{\mathrm{id}}
\newcommand{\KK}{\mathrm{KK}}

\newcommand{\Cs}{C$^\ast$}
\newcommand{\rg}{\mathop{{\mathrm C}_{\mathrm r}^\ast}}

\newcommand{\two}{I\hspace{-1.2pt}I}
\newcommand{\three}{I\hspace{-1.2pt}I\hspace{-1.2pt}I}
\newtheorem{thm}{Theorem}[section]
\newtheorem{prop}[thm]{Proposition}
\newtheorem{cor}[thm]{Corollary}
\newtheorem{lem}[thm]{Lemma}
\newtheorem*{no}{Notation}
\newtheorem*{lem*}{Lemma}
\newtheorem*{thm*}{Theorem}

\newtheorem*{aprop}{Proposition A}
\newtheorem*{alem}{Lemma B}
\theoremstyle{definition}
\newtheorem{defn}[thm]{Definition}
\newtheorem{exa}[thm]{Example}
\newtheorem{rmk}[thm]{Remark}

\title[]{Inclusions of simple \Cs-algebras arising from compact group actions.}
\begin{document}
\author{Miho Mukohara}
\address{Department of Mathematical Sciences, The University of Tokyo, 3-8-1 Komaba, Tokyo, 153-8914
Japan}
\email{miho-mukohara@g.ecc.u-tokyo.ac.jp\\
miho.mukohara@gmail.com}
\date{\today}

\begin{abstract}
Inclusions of operator algebras have long been studied. In particular,
inclusions arising from actions of compact groups on
factors were studied by Izumi-Longo-Popa and others. The
correspondence between intermediate subfactors and subgroups is called
the Galois correspondence.
Analogues for actions on \Cs-algebras have been studied by Izumi,
Cameron-Smith, Peligrad, and others.
In this article,
we give examples of compact group actions on simple \Cs-algebras for which the Galois correspondence holds.
\end{abstract}
\maketitle

\section{Introduction}
Inclusions of \Cs-algebras have been studied in several ways.
When a compact group $G$ acts on a \Cs-algebra $A$,
there exists an inclusion $A^{G}\subset A$.
In the case of von Neumann algebras,
it is a well-known result shown by Izumi-Longo-Popa \cite{ILP} that if the compact group action $G\acts M$ on a factor $M$ is minimal,
then there is a natural bijection from the lattice of intermediate subfactors onto the lattice of closed subgroups of $G$.
More generally,
the intermediate lattices of discrete subfactors are well-studied.
(See \cite{T}, \cite{JP}.)
For \Cs-algebras,
finite index inclusions \cite{I1,W} and inclusions arising from discrete group actions \cite{CS} are well understood.
Moreover,
discrete inclusions of unital \Cs-algebras generated by actions of unital tensor categories are studied by Hern{\'a}ndez Palomares and Nelson \cite{NP}, 
and intermediate discrete inclusions are characterized by algebraic objects.
It is a natural question for a given discrete inclusion of \Cs-algebras,
whether every intermediate inclusion is automatically discrete.
If a compact group $G$ has a minimal action on a factor $M$
(i.e.,
the action is faithful and the relative commutant $M\cap (M^{G})^{\prime}$ is trivial),
then every inclusion $M^{G}\subset N$ is automatically discrete for every intermediate subfactor $M^{G}\subset N\subset M$. 
The goal of this article is to give examples of inclusions of \Cs-algebras such that all intermediate inclusions are discrete.

In Section 4,
we introduce a condition (Condition ($*$)) for inclusions of \Cs-algebras.
An inclusion with Condition ($*$) is a generalization of inclusions arising from discrete group actions such that each non-trivial automorphism is properly outer.
In Theorem 4.10,
under some countability assumptions,
we prove that if a compact group action $G\acts A$ with a simple fixed point algebra induces the inclusion $A^G\subset A$ with Condition ($*$),
then the inclusion $A^G\subset A$ is \Cs-irreducible and the Galois correspondence holds.
Moreover,
we give examples of compact group actions that satisfy the assumptions of Theorem 4.10, such as quasi-product actions with simple fixed point algebras (Corollary \ref{cor ab}),
isometrically shift-absorbing actions (Corollary \ref{cor isa}),
and free product actions (Corollary \ref{cor free}).
The notion of a quasi-product action is introduced in \cite{BEK}
(see Section \ref{seq quasi product}).
This is a class of compact group actions on separable \Cs-algebras whose non-trivial dual endomorphisms are properly outer.
The gauge actions on the Cuntz algebras are examples of quasi-product actions.
Isometrically shift-absorbing actions are introduced by Gabe-Szab\'{o} \cite{GS} for the equivariant classification theory of Kirchberg algebras.
This is an important class of actions on purely infinite \Cs-algebras
(see Section \ref{sec isom shift abs}).

According to \cite[Theorem 1 (8)]{BEK},
isometrically shift-absorbing actions on simple separable \Cs-algebras are quasi-product actions.
Moreover,
after the original version of this paper was submitted, 
Izumi's preprint \cite{I2} appeared on arXiv. 
He gave a remarkable characterization of a
quasi-product action in \cite[Theorem 1.1]{I2},
which gives the following
improvement of our main result when a \Cs-algebra $A$ is separable.
\begin{thm*}[Theorem \ref{main2}]
 Let $\alpha\colon G\acts A$ be a faithful action of a compact second countable group $G$ on a separable \Cs-algebra $A$.
    If the fixed point algebra $A^{G}$ is simple and the relative commutant $M(A)\cap (A^{G})^{\prime}$ is trivial,
    then the map
\[
\Phi\colon\{H\mid H\leq G\}\ni H\mapsto A^{H}\in\{D\mid A^{G}\subset D\subset A\}
\] 
is a bijection from the set of all closed subgroups of $G$ to the set of all intermediate \Cs-subalgebras between $A^G$ and $A$.
\end{thm*}
When $A$ is separable,
Corollaries \ref{cor isa} and \ref{cor free} follow from the above theorem.
However,
since both \cite[Theorem 1]{BEK} and \cite[Theorem 1.1]{I2} are very deep results requiring difficult proofs, 
we provide proofs of Corollaries \ref{cor isa} and \ref{cor free} that are independent of these theorems.

\section{Preliminaries}\label{sec pre}
\subsection{Notations}\label{prep}
We summarize the notations we use throughout this article.
Let $B\subset A$ be a non-degenerate inclusion of $\sigma$-unital \Cs-algebras with the canonical inclusion $\iota\colon B\rightarrow A$ and a conditional expectation $E$ from $A$ onto $B$.
The multiplier algebras of $A$ and $B$ are denoted by $M(A)$ and $M(B)$,
respectively. 
We use the following notations.
\begin{itemize}
\item The extended conditional expectation from $M(A)$ onto $M(B)$ is also denoted by $E$.
\item Let $\cE_{E}=(\cE_{E}, \langle,\rangle)$ be the canonical Hilbert \Cs-module associated with $E$,
$\eta\colon A\rightarrow \cE_{E}$ be the canonical inclusion,
(i.e., $\langle\eta(a),\eta(b)\rangle=E(a^{*}b)$ for every $a,b\in A$).
The symbol $\cL(\cE_{E})$ denotes the \Cs-algebra of all adjointable operators of $\cE_{E}$.
\item There is the natural left action $\phi\colon A\rightarrow\cL(\cE_{E})$ of $A$ on $\cE_{E}$.
For any $a\in A$ and $f\in\cE_{E}$,
we write $af$ instead of $\phi(a)f$ if no confusion arises.
\item The Jones projection with respect to $E$ is denoted by $e\in\cL(\cE_{E})$.
(That is, $e\eta(a):=\eta(E(a))$ for all $a\in A$.)
\item The set of compact operators of $\cE_{E}$ is denoted by $\cK(\cE_{E})$ or $A_1$.
We have $A_1=\overline{\text{span}}\{aeb\mid a, b\in A\}\subset\cL(\cE_{E})$.
\item For a compact group $G$,
let $\lambda$ be the left regular representation, $\widehat{G}$ (resp. ${\rm Rep}_{\rm f}(G)$) be the set of all equivalence classes of irreducible (resp. finite dimensional) representations, $1_{G}$ be the unit of $G$, and $1_{\widehat{G}}$ be the trivial representation.
In this paper, the measure on G is always the normalized Haar measure.
\item For any $\epsilon>0$ and any elements $a$, 
$b$ of a normed space $X$,
we write $a\approx_{\epsilon}b$ for $\|a-b\|<\epsilon$.
\end{itemize}

\subsection{Finite index inclusions of simple \Cs-algebra}\label{sec fin}
The Watatani index of unital inclusions of \Cs-algebras is introduced by Watatani in \cite{W}. 
We introduce the definition of the Watatani indices for nonunital inclusions of \Cs-algebras based on \cite{I1}.
Let $B\subset A$ be a non-degenerate inclusion of $\sigma$-unital \Cs-algebras with a conditional expectation $E$ from $A$ onto $B$.
\begin{defn}
The Pimsner-Popa index $\mathrm{Ind}_{\mathrm{p}}\:E$ of $E$ is defined as follows:
\[
\mathrm{Ind}_{\mathrm{p}}\:E:=\inf\{\lambda>0\mid\lambda E-\id_{A}\text{ is completely positive.}\}.
\]
\end{defn}
\begin{defn}[Theorem 2.8 of \cite{I1}]\label{def_Watatani}
When $\mathrm{Ind}_{\mathrm{p}}\:E<\infty$ and $A$ is contained in $A_1$ as a subalgebra of $\cL(\cE_{E})$,
there is a bounded completely  positive $A$-$A$-bimodule map $\hat{E}$ from $\cL(\cE_{E})$ onto $M(A)$ such that $\hat{E}(e)=1$ and $\hat{E}(1)\in Z(M(A))$.
The Watatani index of $E$ is defined by 
\begin{equation*}
\mathrm{Ind}_{\mathrm{w}}\: E:=
\begin{cases}
\hat{E}(1)&\text{if $A\subset A_1$, }\\
\infty&\text{otherwise.}
\end{cases}
\end{equation*}
\end{defn}
\begin{defn}\label{quasi basis}
If there exists a family $\{(v_{i}, u_{i})\}_{i=1}^{n}$ in $M(A)\times M(A)$ such that $x=\sum_{i=1}^{n}E(xv_{i})u_{i}=\sum_{i=1}^{n}v_{i}E(u_{i}x)$ for every $x\in A$,
then $\{(v_{i}, u_{i})\}_{i=1}^{n}$ is called a quasi basis of $E$.
\end{defn}

In general,
if there is a quasi basis $\{(v_{i}, u_{i})\}_{i=1}^{n}$,
then $1=\sum_{i=1}^n v_i e u_i$ holds in $\cL(\cE_{E})$ and $\mathrm{Ind}_{\mathrm{p}}\:E\leq\|\sum_{i=1}^n v_i u_i\|$ according to Proposition 2.6.2 of \cite{W}.
Consequently,
the Watatani index is given by 
$\mathrm{Ind}_{\mathrm{w}}\:E = \hat{E}(\sum_{i=1}^n v_i e u_i)=\sum_{i=1}^n v_i u_i$.
When an inclusion $B\subset A$ is unital,
the existence of a quasi basis $\{(v_{i}, u_{i})_{i=1}^{n}\}$ is equivalent to $E$ having a finite Watatani index by Proposition 2.1.5 of \cite{W} and Definition \ref{def_Watatani}.
On the other hand,
if $A$ and $B$ are non-unital,
having a finite Watatani index does not imply the existence of a quasi basis \cite[Example 2.15]{I1}.
However,
the following theorem is true for inclusions of simple stable \Cs-algebras.

\begin{thm}[Proposition 3.6 of \cite{I1}]\label{p=e}
If $A$ and $B$ are simple stable \Cs-algebras and we have $\mathrm{Ind}_{\mathrm{p}}\: E=d<\infty$,
then there is an isometry $W\in M(A)$ such that the pair $(d^{\frac{1}{2}}W^{*}, d^{\frac{1}{2}}W)$ is a quasi basis of $E$.
In particular,
in this case,
we get $A\subset A_1$ and $\mathrm{Ind}_{\mathrm{p}}\:E=\mathrm{Ind}_{\mathrm{w}}\:E$.
\end{thm}
Due to the above theorem,
we write $\mathrm{Ind}\:E$ instead of $\mathrm{Ind}_{\mathrm{w}}\:E$  if $A$ and $B$ are simple and stable. 

\begin{exa}[Corollary 3.12 of \cite{I1}]\label{finitegroup}
    Let $\Gamma\acts A$ be an outer action of a finite group on a simple \Cs-algebra and $E\colon A\rightarrow A^{G}$ be the canonical conditional expectation.
    Then,
    the Watatani index is given by ${\rm Ind}_{\rm w}\: E=|\Gamma|$ and we have $A_{1}=A\rtimes\Gamma$.
    In this case,
    the completely positive map $\hat{E}$ satisfies
    \[
    \hat{E}(\sum_{s\in\Gamma}a_{s}\lambda_{s})=|\Gamma|a_{1_{\Gamma}}.
    \]
    This equation is trivial for $\sum_{s\in\Gamma}a_{s}\lambda_{s}=aeb$,
    with $a,b\in A$ and $e=\frac{1}{|\Gamma|}\sum_{s}\lambda_{s}$.
    Since we have $A_{1}=\overline{\rm span}\:AeA$,
    the above equation holds for every $\sum_{s\in\Gamma}a_{s}\lambda_{s}\in A_{1}$. 
\end{exa}

Next,
we explain finite index endomorphisms and sectors.
(See section 4 of \cite{I1} for details.)
Let $A$ and $B$ be stable simple \Cs-algebras.
We use the following notations.
\begin{itemize}
\item A $*$-homomorphism $\rho\colon A\rightarrow B$ is called irreducible,
if the relative commutant $M(B)\cap\rho(A)^{\prime}$ is trivial.
\item A $*$-homomorphism $\rho\colon A\rightarrow B$ is said to have a finite index if there exists a finite index conditional expectation $E\colon B\rightarrow\rho(A)$.
It is known that there is a unique conditional expectation $E_{\rho}\colon B\rightarrow\rho(A)$ which has a minimal index
(see Theorem 2.12.3 of \cite{W} and Section 3 of \cite{I1}).
The square root $\sqrt{\mathrm{Ind}_{\mathrm{w}}\: E_{\rho}}$ of the minimal index is denoted by $d(\rho)$.
\item  An intertwiner space 
\[
\{T\in M(B)\mid T\rho_{1}(x)=\rho_{2}(x)T,\mathrm{\ for\ all\ } x\in A\}
\]
of $*$-homomorphisms $\rho_{1}, \rho_{2}\colon A\rightarrow B$ is denoted by $(\rho_{1}, \rho_{2})$.
If there is a unitary $U\in(\rho_{1}, \rho_{2})$,
then two $*$-homomorphisms $\rho_{1}$ and $\rho_{2}$ are said to be equivalent and denoted by $\rho_{1}\sim\rho_{2}$.
\item We write $[\rho]$ for the equivalence class of an endomorphism $\rho\colon B\rightarrow B$.
\item The set of equivalence classes $\mathrm{End}(B)/\sim$ is denoted by $\mathrm{Sect}(B)$.
Since $B$ is stable,
there are isometries $S_{1}, S_{2}\in M(B)$ with $S_{1}S_{1}^{*}+S_{2}S_{2}^{*}=1_{M(B)}$.
The product and the sum of $[\rho_1]$ and $[\rho_2]$ in $\mathrm{Sect}(B)$ are defined by $[\rho_{1}][\rho_{2}]:=[\rho_{1}\circ\rho_{2}]$ and $[\rho_{1}]\oplus[\rho_{2}]:=[\rho]$,
where $\rho$ is an endomorphism of $B$ such that
\[
\rho(x)=S_{1}\rho_{1}(x)S_{1}^{*}+S_{2}\rho_{2}(x)S_{2}^{*}
\] for all $x\in B$.
\end{itemize}

By Lemma 4.1 of \cite{I1},
every finite index endomorphism of a $\sigma$-unital simple stable \Cs-algebra is decomposed into a direct sum of irreducible endomorphisms.
For any finite index $*$-homomorphism $\rho\colon A\rightarrow B$ between $\sigma$-unital simple stable \Cs-algebras,
there is a conjugate $*$-homomorphism $\overline{\rho}\colon B\rightarrow A$ as follows.
\begin{lem}[Lemma 4.4 of \cite{I1}]\label{conjugate}
Let $\rho\colon A\rightarrow B$ be a finite index $*$-homomorphism between $\sigma$-unital simple stable \Cs-algebras.
Then there is a finite index endomorphism $\overline{\rho}$,
unique up to equivalence,
such that the following holds:
there exist isometries $R_{\rho}\in(\id_{A}, \overline{\rho}\circ\rho)$ and $\overline{R}_{\rho}\in(\id_{B}, \rho\circ\overline{\rho})$ such that
\[
\overline{R}_{\rho}^{*}\rho(R_{\rho})=R_{\rho}^{*}\overline{\rho}(\overline{R}_{\rho})=\frac{1}{d(\rho)}.
\]
Moreover,
$d(\rho)=d(\overline{\rho})$ holds,
and for any $*$-homomorphisms $\pi_{1}:A\rightarrow C$ and $\pi_{2}:B\rightarrow C$,
we have a linear isomorphism
\[
(\pi_{1}, \pi_{2}\circ\rho)\ni T\mapsto \pi_{2}(\overline{R}^{*}_{\rho})T\in(\pi_{1}\circ\overline{\rho}, \pi_{2}).
\]
\end{lem}

\subsection{Compact group actions on non-unital \Cs-algebras}
In this subsection,
we discuss the following general properties of compact group actions on \Cs-algebras.
\begin{lem}
A norm continuous action $\alpha$ of a compact group $G$ on a \Cs-algebra $A$ extends to a strict continuous action of $G$ on $M(A)$.  
\end{lem}
\begin{proof}
By Proposition 3.12.10 of \cite{Ped},
each $\alpha_{g}$ extends to a strictly continuous $*$-homomorphism between $M(A)$.
By straightforward computations,
the map $G\ni g\mapsto\alpha_{g}(x)\in M(A)$ is continuous with respect to the strict topology for any $x\in M(A)$.
\end{proof}
In this article,
suppose that the extended action $G\acts M(A)$ is also denoted by $\alpha$.
\begin{lem}[see Lemma 2.6 \cite{I1}]
Let $\alpha\colon G\acts A$ be an action of a compact group,
then the following hold.
\begin{itemize}
\item[(1)] We have the inclusion $M(A^{G})\subset M(A)$ as \Cs-subalgebras of $A^{**}$.
\item[(2)] The canonical conditional expectation $E\colon A\rightarrow A^{G}$ extends to the strictly continuous conditional expectation from $M(A)$ onto $M(A^{G})$.
\item[(3)] We have $M(A)^{G}=M(A^{G})$.
\end{itemize}
\end{lem}
\begin{proof}
For (1),
it suffices to construct approximate units $(a_{\lambda})_{\lambda}$ of $A$ which are contained in $A^{G}$.
This is trivial because ($E(x_{\lambda}))_\lambda$ are approximate unit of $A$ if $(x_{\lambda})_\lambda$ are approximate units of $A$.

For (2),
we show that the restriction $\tilde{E}$ of $E^{**}\colon A^{**}\rightarrow (A^{G})^{**}$ on $M(A)$ is a strictly continuous conditional expectation from $M(A)$ onto $M(A^{G})$. 
By Theorem 3.12.9 of \cite{Ped},
we have $M(A)_{sa}=(\tilde{A}_{sa})^{m}\cap(\tilde{A}_{sa})_{m}$.
Thus,
$E^{**}(M(A))=M(A^{G})$ holds.
The rest of the statement follows from a simple calculation.

Since $\varphi(E^{**}(x))=\int_{G}\varphi(\alpha_{g}(x))dg$ for any $x\in M(A)$ and $\varphi\in A^{*}$,
we have $M(A)^{G}\subset E^{**}(M(A))=M(A^{G})$.
The reverse inclusion follows from the strict continuity of each $\alpha_{g}\colon M(A)\rightarrow M(A)$.
\end{proof}

\subsection{\Cs-valued weights} 
Let $\Gamma\acts A$ be a finite group action and $E\colon A\rightarrow A^{\Gamma}$ be the canonical conditional expectation.
As in Example \ref{finitegroup},
the completely positive map $\hat{E}$ in Definition \ref{def_Watatani} can sometimes be expressed as a scalar multiple of the canonical conditional expectation from $A\rtimes\Gamma$ onto $A$.
For an infinite compact group action $G\acts A$,
the conditional expectation $A\rightarrow A^{G}$ does not have a finite Watatani index,
and the evaluation map $C(G,A)\ni f\mapsto f(1_{G})\in A$ does not extend to a bounded map from $A\rtimes G$. 
To discuss unbounded completely positive maps between \Cs-algebras,
we introduce the concept of a \Cs-valued weight.
The definition of a \Cs-valued weight is introduced by Kustermans as an analogue of an operator valued weight
(see \cite{H1} for details on operator valued weights).
\begin{defn}[Definition 1.1 of \cite{Ku}]
Consider two \Cs-algebras $A, B$ and a hereditary cone $P$ in $A_{+}$.
Put $\cN=\{a\in A\mid a^{*}a\in P\}$ and $\cM=\mathrm{span}\;P=\mathrm{span}\;\cN^{*}\cN$.
Suppose that $\varphi$ is a linear mapping from $\cM$ into $M(B)$ such that 
\[
\sum_{i,j=1}^{n}b_{j}^{*}\varphi(a_{j}^{*}a_{i})b_{i}\geq0
\]
for all $n\in\IN$ and all $a_{1},\dots, a_{n}\in \cN$, 
$b_{1},\dots, b_{n}\in B$.
Then we call $\varphi$ a \Cs-valued weight from $A$ into $M(B)$.
In this case,
$\cN, \cM$ are denoted by $\cN_{\varphi}, \cM_{\varphi}$.
\end{defn}
When a locally compact group $G$ acts on a von Neumann algebra $M$,
there is the canonical operator valued weight from $M\bar{\rtimes}G$ onto $M$
(see Theorem 3.1 of \cite{H3}).
Similarly,
there are canonical \Cs-valued weights associated with group actions.
\begin{exa}\label{dualweight}
Suppose $\alpha\colon G\acts A$ is a continuous action of a compact group $G$ on a \Cs-algebra.
Let $\mathcal{P}$ be the set of minimal projections of $\rg(G)\subset M(A\rtimes G)$.
It is well-known that $\rg(G)$ is of type I,
and for each $p\in \mathcal{P}$,
there is an irreducible representation $(\pi, V_{\pi})$ of $G$ and a vector $\xi\in V_{\pi}$ such that $p=\int_{G}\langle\pi(g)\xi,\xi\rangle\lambda_{g}dg$ holds.
Set an algebraic left ideal
\[
\cN_{\hat{E}}:=\mathrm{span}\{xpa\in M(A\rtimes G)\mid x\in M(A\rtimes G), p\in\mathcal{P}, a\in M(A)\}
\]
and a algebraic hereditary subalgebra
\[
\cM_{\hat{E}}:=\mathrm{span}\{bqxpa\in M(A\rtimes G)\mid x\in M(A\rtimes G), p, q\in \mathcal{P}, a,b\in M(A)\}.
\]
We get the canonical $\sigma$-weak continuous action 
$G\acts A_{\alpha}^{\prime\prime}$ on the von Neumann algebra,
where $A_{\alpha}^{\prime\prime}:=\{x\in A^{**}\mid G\ni g\mapsto\alpha^{**}_{g}(x)\in A^{**}\mathrm{\ is\ }\sigma\mathrm{-weak\ 
continuous}\}$.
Let $T$ be the canonical operator valued weight from $(A_{\alpha}^{\prime\prime}\bar{\rtimes}G)_{+}$ to $A_{\alpha+}^{\prime\prime}$.
We use the notations $\mathfrak{m}_{T}$ and $\mathfrak{n}_{T}$ as in Section 2 of \cite{H1}.
The map $T$ extends to a linear map from the domain $\mathfrak{m}_{T}\subset A_{\alpha}^{\prime\prime}\bar{\rtimes}G$ onto $A_{\alpha}^{\prime\prime}$. 
We claim that the restriction $\hat{E}$ of $T$ to the domain $\cM_{\hat{E}}$ is a \Cs-valued weight from $M(A\rtimes G)$ into $M(A)$.
To show this,
we have to prove $\cM_{\hat{E}}\subset\mathfrak{m}_{T}$ and $T(\cM_{\hat{E}})\subset M(A)$.
By Theorem 3.1 (c) of \cite{H3},
$pa$ is contained in $\mathfrak{n}_{T}$ for every $p\in\mathcal{P}$ and every $a\in M(A)$.
Since $\mathfrak{m}_{T}=\rm{span\;}\mathfrak{n}_{T}^{*}\mathfrak{n}_{T}$ is a hereditary subalgebra of $A_{\alpha}^{\prime\prime}\bar{\rtimes}G$,
we get $\cM_{\hat{E}}\subset\mathfrak{m}_{T}$.
Take $p,q\in \mathcal{P}$,
$x\in M(A\rtimes G)$,
$a,b\in M(A)$,
$c\in A$,
and $\epsilon>0$.
Since $ac\in A$,
there is an element $d\in A^{G}$ such that $bT(qxp)ac\approx_{\epsilon}bT(qxp)dac$.
Thanks to Section 2 (iii) of \cite{H1},
we get
\begin{align*}
    T(bqxpa)c
   =bT(qxp)ac
   \approx_{\epsilon}bT(qxp)dac
   =bT(q(xdp)p)ac.
\end{align*}
Since $xdp\in A\rtimes G$ and the linear map $A\rtimes G\ni y\mapsto T(qyp)\in A_{\alpha}^{\prime\prime}$ is continuous,
there is an element $x_{0}=\int_{G}f(g)\lambda_{g}dg$ with $f\in C(G,A)$ such that $bT(q(xdp)p)ac\approx_{\epsilon} bT(qx_{0}p)ac$.
By Lemma 2.5 of \cite{H2} 
and Theorem 3.1 (d) of \cite{H3},
we get $T(qx_{0}p)\in A$.
This implies $T(bqxpa)c\in A$ for every $c\in A$.
Similarly,
we get $cT(bqxpa)\in A$ for every $c\in A$ and $T(bqxpa)\in M(A)$.
This proves $T(\cM_{\hat{E}})\subset M(A)$.
Since $\sum_{i,j=1}^{n}b_{j}^{*}T(a_{j}^{*}a_{i})b_{i}=T((\sum_{j}a_{j}b_{j})^{*}(\sum_{i}a_{i}b_{i}))\geq0$ holds for all $a_{1},\dots, a_{n}\in \cN_{\hat{E}}$ and
$b_{1},\dots, b_{n}\in A$,
the restriction $\hat{E}:=T|_{\cM_{\hat{E}}}$ is a \Cs-valued weight from $M(A\rtimes G)$ to $M(A)$.
If $G$ is a finite group with the normalized Haar measure $\frac{1}{|G|}\sum_{g\in G}\delta_{g}$ and $G\acts A$ is an outer action on a simple \Cs-algebra,
then $\hat{E}$ is equal to the completely positive map in Definition \ref{def_Watatani} by Example 2.5 and Corollary 3.7 of \cite{H3}.
\end{exa}

\begin{lem}\label{lem dual}
Let $\hat{E}$ be the \Cs-valued weight from $\cM_{\hat{E}}$ to $M(A)$ defined in Example \ref{dualweight}.
Then, the following hold.
\begin{itemize}
\item[(1)] $\hat{E}(axb)=a\hat{E}(x)b$ for any $a, b\in M(A), x\in \cM_{\hat{E}}$ .
\item[(2)] Let $p_{1_{\widehat{G}}}\in\mathcal{P}$ be the averaging projection defined as $p_{1_{\widehat{G}}}:=\int_{G}\lambda_{g}dg$. 
We have $p_{1_{\widehat{G}}}\hat{E}(p_{1_{\widehat{G}}}y)=p_{1_{\widehat{G}}}y$ for all $y\in \cN_{\hat{E}}$.
\end{itemize}
\end{lem}

\begin{proof}
(1) follows from Section 2 (iii) of \cite{H1}.

To show (2),
consider $x\in M(A\rtimes G)$,
$b\in M(A)$,
$q\in\mathcal{P}$,
and $a\in A^{G}$.
Since $p_{1_{\widehat{G}}}axq$ is contained in $A\rtimes G$,
we can take a sequence $(f_{n})_{n=1}^{\infty}$ of $C(G,A)$ and elements $x_{n}:=\int_{G}f_{n}(g)\lambda_{g}dg$ of $A\rtimes G$ such that $\lim_{n\to\infty}x_{n}=p_{1_{\widehat{G}}}axq$ and $x_{n}=p_{1_{\widehat{G}}}x_{n}q$ hold for every $n$.
Then,
we have
\[
p_{1_{\widehat{G}}}x_{n}q=p_{1_{\widehat{G}}}\int_{G}f_{n}(g)\lambda_{g}dg=p_{1_{\widehat{G}}}\int_{G}\alpha_{g^{-1}}(f_{n}(g))dg
=p_{1_{\widehat{G}}}\hat{E}\left(p_{1_{\widehat{G}}}x_{n}q\right)
\]
for every $n$.
Thus,
we get
\begin{align*}
    ap_{1_{\widehat{G}}}xqb&
    =\lim_{n\to\infty}p_{1_{\widehat{G}}}x_{n}qb\\
    &=\lim_{n\to\infty}p_{1_{\widehat{G}}}\hat{E}\left(p_{1_{\widehat{G}}}x_{n}qb\right)\\
    &=p_{1_{\widehat{G}}}\hat{E}(p_{1_{\widehat{G}}}axqb)\\
    &=ap_{1_{\widehat{G}}}\hat{E}(p_{1_{\widehat{G}}}xqb)
\end{align*}
for every $a\in A^G$.
The third equality follows from the continuity of the map $M(A\rtimes G)\ni z\mapsto\hat{E}(p_{1_{\widehat{G}}}zq)$.
This implies $p_{1_{\widehat{G}}}xqb=p_{1_{\widehat{G}}}\hat{E}(p_{1_{\widehat{G}}}xqb)$.
By the definition of $\cN_{\hat{E}}$,
we get (2).
\end{proof}
\subsection{Isometrically shift-absorbing actions}\label{sec isom shift abs}
Isometrically shift-absorbing actions are introduced by Gabe and Szab\'{o} \cite{GS}.
In this subsection,
let $\alpha\colon G\acts A$ be an action of a locally compact second countable group $G$ on a separable \Cs-algebra $A$.
\begin{defn}
Let $l_{\alpha}^{\infty}(\IN, A)$ be the \Cs-algebra of every bounded sequence $(a_n)_n$ such that the map $G\ni g\mapsto (\alpha_{g}(a_n))_n\in l^{\infty}(\IN, A)$ is continuous.
The quotient $l_{\alpha}^{\infty}(\IN, A)/c_{0}(\IN, A)$ is denoted by $A_{\infty, \alpha}$ and the action $G\acts A_{\infty, \alpha}$ associated with $\alpha$ is denoted by $\alpha_{\infty}$.
We write $F_{\infty, \alpha}(A)$ for the central sequence algebra $(A_{\infty, \alpha}\cap A^{\prime})/(A_{\infty, \alpha}\cap A^{\perp})$ and $\tilde{\alpha}_{\infty}\colon G\acts F_{\infty, \alpha}(A)$ for the action associated with $\alpha_{\infty}$.
\end{defn}

\begin{defn}[Definition 3.7 of \cite{GS}]
An action $\alpha\colon G\acts A$ is isometrically shift-absorbing if there is a linear map $\fs\colon L^{2}(G)\rightarrow F_{\infty, \alpha}(A)$ such that $\fs(\lambda_{g}(\xi))=\tilde{\alpha}_{\infty,g}(\fs(\xi))$ and $\fs(\xi)^{*}\fs(\zeta)=\langle\zeta,\xi\rangle$ hold for any $\xi, \zeta\in L^{2}(G)$,
and any $g\in G$.
\end{defn}
Gabe and Szab\'{o} proved that for any amenable action $\alpha\colon G\acts A$ of a locally compact second countable group $G$ on a separable nuclear \Cs-algebra,\
there is an amenable action $\beta$ on some Kirchberg algebra which is isometrically shift-absorbing and $\KK^{G}$-equivalent to $\alpha$ (Theorem 3.13 of \cite{GS}). 
When $G$ is discrete and $A$ is a Kirchberg algebra,
an action $\alpha\colon G\acts A$ is isometrically shift-absorbing if and only if it is pointwise outer (Theorem 3.15 of \cite{GS}).

In this article,
we use the following.

\begin{thm}[Proposition 3.8 of \cite{GS}]\label{char}
An action $\alpha:G\acts A$ is isometrically shift-absorbing if and only if there is a $G$-equivalent $A$-bimodule map $\theta\colon(L^{2}(G, A), \lambda\otimes\alpha)\rightarrow(A_{\infty, \alpha}, \alpha_{\infty})$ such that $\theta(\xi)^{*}\theta(\zeta)=\langle\xi,\zeta\rangle$ for all $\xi, \zeta\in L^{2}(G, A)$,
where $\langle,\rangle$ is an inner product defined by $\langle\xi,\zeta\rangle:=\int_{G}\xi(g)^{*}\zeta(g)dg$.
\end{thm}


\begin{exa}[Definition 3.4 of \cite{GS}]\label{isa exa}
Let $G$ be a compact second countable group and $\gamma\colon G\acts\cO_{\infty}$ be a quasi-free action on the Cuntz algebra induced by the countable infinite repeat $\lambda\otimes\mathrm{id}\colon G\acts L^{2}(G)\otimes l^{2}(\IN)$ of the left regular representation $\lambda$
(i.e., there is a liner map $L^{2}(G)\otimes l^{2}(\IN)\ni\xi\mapsto\hat{\xi}\in\cO_{\infty}$ such that $\cO_{\infty}=\text{C}^{*}(\{\hat{\xi}\}_{\xi\in L^{2}(G)\otimes l^{2}(\IN)})$,
$\hat{\eta}^{*}\hat{\xi}=\langle\xi,\eta\rangle$ and $\gamma_{g}(\hat{\xi})=\widehat{\lambda_{g}\otimes\mathrm{id}(\xi)}$ hold for every $\xi, \eta\in L^{2}(G)\otimes l^{2}(\IN)$ and $g\in G$).
Then the infinite tensor product action $\gamma^{\otimes\infty}\colon G\acts \cO_{\infty}^{\otimes\infty}$ is isometrically shift-absorbing.
\end{exa}
\subsection{Properly outer endomorphisms and quasi-product actions}\label{seq quasi product}
The notion of a properly outer automorphism of a von Neumann algebra is introduced by Connes.
The analogue for automorphisms of \Cs-algebras is also well-known
(see \cite{Ki}, \cite[Section 6]{OP}).
In this article,
we consider properly outer endomorphisms of simple \Cs-algebras.
We use the following definition
(see \cite{BEK},\cite[Section 7]{I1}, \cite[Section 3]{I2} for details).
\begin{defn}[See Section 1 of \cite{BEK}]\label{defn prop outer}
    An endomorphism $\rho$ of a \Cs-algebra $A$ is properly outer if for every nonzero hereditary subalgebra $H$ of $A$,
    every $x\in A$,
    and every $\epsilon>0$,
    there exists a positive element $c$ in $H$ with $\|c\|=1$ satisfying $\|cx\rho(c)\|<\epsilon$.
\end{defn}

We mention the following lemma.
The proof of \cite[Lemma 3.2]{Ki} works for properly outer endomorphisms.
\begin{lem}[Lemma 3.2 of \cite{Ki}]\label{lem prop outer}
Let $a$ be a positive element of a \Cs-algebra $A$,
$x_{1},\dots,x_{n}$ be elements of $A$,
and $\rho_{1},\dots,\rho_{n}$ be properly outer endomorphisms of $A$.
Then,
for every $\epsilon>0$,
there is a positive element $c\in A$ with $\|c\|=1$ such that 
\[
\|cac\|>\|a\|-\epsilon,\;{\rm and}\;\|cx_{i}\rho_{i}(c)\|<\epsilon
\]
hold for every $i=1,\dots,n$.
\end{lem}

The notion of quasi-product actions is introduced by Bratteli-Elliott-Kishimoto in \cite{BEK}.
They proved several equivalent conditions for compact group actions on separable prime \Cs-algebras.
In this article,
we use the following definition.
\begin{defn}[Theorem 1 (3) of \cite{BEK}]\label{def quasi product}
    Suppose $\alpha\colon G\acts A$ is an action of a compact second countable group on a separable \Cs-algebra. 
    For each irreducible representation $\sigma$ of $G$,
    let $\hat{\alpha}_{\sigma}$ be the dual endomorphism of the stabilized crossed product $A\rtimes G\otimes\IK$
    (see Section 1 of \cite{BEK} for the definition).
    We say that $\alpha$ is a quasi-product action if $A\rtimes G$ is prime and $\hat{\alpha}_{\sigma}$ is properly outer for every non-trivial irreducible representation $\sigma$.
\end{defn} 
 We note that separability is required for Theorem 1 of \cite{BEK} to hold.
 \begin{rmk}\label{rmk prime}
     If an action $\alpha\colon G\acts A$ is quasi-product,
     we have $M(A)\cap (A^{G})^{\prime}=\IC$.
     Indeed,
     if $M(A)\cap (A^{G})^{\prime}\neq\IC$,
     for every faithful representation $\pi\colon A\rightarrow B(\cH)$,
     $B(\cH)\cap\pi(A^{G})^{\prime}\neq\IC$ holds.
     This contradicts Theorem 1 (2) of \cite{BEK}.
 \end{rmk}
 \begin{rmk}
     By Theorem 1 (8) of \cite{BEK},
     an isometrically shift-absorbing action $G\acts A$ on a separable prime \Cs-algebra is quasi-product.
 \end{rmk}

\section{Irreducible decomposition of Hilbert \Cs-bimodules}\label{sec decomp}
When a finite group action $\alpha\colon\Gamma\acts A$ on a $\sigma$-unital simple \Cs-algebra is outer and stable,
the canonical Hilbert \Cs-module $\cE_{E}$ associated with the conditional expectation $E\colon A\rightarrow A^{\Gamma}$ has an irreducible decomposition as an $A^{\Gamma}$-bimodule.
Each irreducible direct summand corresponds to an irreducible representation of $\Gamma$
(see Section 6 of \cite{I1}).
The goal of this section is to give a similar statement for compact group actions under certain assumptions.
In Section 4 of \cite{I1},
the sectors associated with finite index inclusions \cite[Definition 4.5]{I1} are discussed.
Similarly,
we consider the sectors associated with inclusions arising from compact group actions.
The families of the endomorphisms $\{\rho_{\sigma}\}_{\sigma}$ and the Hilbert spaces $\{\cH_{\sigma}\}_{\sigma}$ to be constructed in this section are well-known to experts and may not require detailed explanation for some readers.
(See Section 2.4 of \cite{I2},
Section \three.2 of \cite{AHKT} and Section 4 of \cite{Ro}.)
In this section, 
for the reader's convenience,
we detail a specific construction of $\{\rho_{\sigma}\}_{\sigma}$ and $\{\cH_{\sigma}\}_{\sigma}$.

Throughout this section,
let $A$ be a $\sigma$-unital simple \Cs-algebra, $G$ be a compact second countable group and $(A, G, \alpha)$ be a \Cs-dynamical system satisfying the following.
\begin{itemize}
\item The action $\alpha$ is faithful and stable,
(i.e., $(A, \alpha)$ is conjugate to $(A\otimes\IK, \alpha\otimes{\rm id})$).
\item The fixed point algebra $A^G$ is simple.
\item The relative commutant $M(A)\cap(A^{G})^{\prime}$ is trivial. 
\end{itemize}
By Proposition A in the Appendix section,
the crossed product $A\rtimes G$ is simple under the above assumptions.
Let $E\colon A\rightarrow A^{G}$ be the canonical conditional expectation.
Additionally,
we use the notation introduced in Section \ref{sec pre} below.
\begin{lem}
There is an isomorphism between $\cK(\cE_{E})$ and $A\rtimes G$.
\end{lem}
\begin{proof}
There are the natural $*$-homomorphism $\phi\colon A\rightarrow\cL(\cE_{E})$ and the unitary representation $u\colon G\rightarrow\cL(\cE_{E})$ defined by 
\[
\phi(a)\eta(b):=\eta(ab),\;u_{g}(\eta(a)):=\eta(\alpha_{g}(a))
\]
for all $a, b\in A$ and all $g\in G$.
Since the pair $(\phi, u)$ generates a covariant representation of $(A, G, \alpha)$,
we have the natural $*$-homomorphism $\phi\rtimes u\colon A\rtimes G\rightarrow \cL(\cE_{E})$.
Let $p_{1_{\widehat{G}}}:=\int_{G}\lambda_{g}dg\in M(A\rtimes G)$,
then we get $\phi\rtimes u(p_{1_{\widehat{G}}})=e$,
where $e$ is the Jones projection.
This implies $\cK(\cE_{E})=\overline{\mathrm{span}}\{aeb\mid a, b\in A\}\subset\phi\rtimes u(A\rtimes G)$.
Since $A\rtimes G$ is simple and $\cK(\cE_{E})$ is an ideal of $\cL(\cE_{E})$,
$\phi\rtimes u$ is an isomorphism between $A\rtimes G$ and $\cK(\cE_{E})$. 
\end{proof}
Using the above isomorphism,
the Jones projection $e$ is identified with the averaging projection $p_{1_{\widehat{G}}}=\int_{G}\lambda_{g}dg\in M(A\rtimes G)\cap (A^{G})^\prime$.
To construct the canonical isomorphism $\gamma_{1}$ as in Lemma 4.2 of \cite{I1},
we use the following result.
\begin{thm}[Theorem 4.23 of \cite{B}]\label{Brown}
Let $C$ be a $\sigma$-unital \Cs-algebra,
and $p\in M(C)$ be a projection.
If the hereditary \Cs-subalgebra $A:=pCp$ generated by $p$ is stable and generates $C$ as an ideal,
then there exists an isometry $u\in M(C)$ with $uu^{*}=p$.
\end{thm}
\begin{lem}\label{canonical}
There is an isomorphism $\gamma_{1}\colon A\rtimes G\rightarrow A^{G}$.
\end{lem}
\begin{proof}
Since the corner \Cs-algebra $e(A\rtimes G)e=eA^{G}$ is stable and full,
we have an isometry $v\in M(A\rtimes G)$ such that $vv^{*}=e$ by Theorem \ref{Brown}.
Then there is an $*$-isomorphism $\gamma_{1}\colon A\rtimes G\rightarrow A^{G}$ such that $vxv^{*}=\gamma_{1}(x)e$ for any $x\in A\rtimes G$.
\end{proof}

It is well-known that $A\rtimes G$ is isomorphic to $(A\otimes \IK(L^{2}(G)))^{G}$,
where $A$ has the action $\alpha$ and $\IK(L^2(G))$ has the action associated with the right regular representation.
Since the relative commutant $M(A)\cap(A^{G})^\prime$ is trivial,
we have isomorphisms 
\[
M(A\rtimes G)\cap (A^{G})^{\prime}\cong M(A\otimes \IK(L^{2}(G)))^{G}\cap (A^{G}\otimes 1)^{\prime}=1\otimes\rg(G)^{\prime\prime}\cong\prod_{[\sigma]\in\widehat{G}}B(V_{\sigma}),
\]
where each $V_\sigma$ is a representation space of $\sigma$. 
Since the relative commutant $M(A\rtimes G)\cap (A^{G})^{\prime}$ is of type $I$ and $A^G$ is stable,
we get a system of Hilbert spaces $\{\cH_{\sigma}\}_{[\sigma]\in\widehat{G}}$ in $M(A)$ as follows.
(See Section 4 of \cite{I1}.)

Fixing the system of representatives,
we assume that $\widehat{G}$ is the set of pairwise orthogonal irreducible representations of $G$.
Let $A_{\sigma}$ be a direct summand of $M(A\rtimes G)\cap (A^{G})^{\prime}$ which corresponds to $\sigma\in\widehat{G}$.
Take a minimal projection $p_{\sigma}\in A_{\sigma}$.
Since we have $p_{\sigma}(A\rtimes G)p_{\sigma}\supset p_{\sigma}A^{G}\cong A^G$ and $A^{G}$ is stable,
there is a partial isometry $w_{\sigma}\in M(A\rtimes G)$ such that $w_{\sigma}w_{\sigma}^{*}=e$ and $w_{\sigma}^{*}w_{\sigma}=p_{\sigma}$ by Theorem \ref{Brown}.
Since $e=p_{1_{\widehat{G}}}$ and $p_{\sigma}$ are elements in $\mathcal{P}$
(see Example \ref{dualweight} for the definition),
and $w_{\sigma}=ew_{\sigma}p_{\sigma}$,
$w_{\sigma}$ is contained in the domain $\cM_{\hat{E}}$ of $\hat{E}$ for every $\sigma\in\widehat{G}$.
We define the endomorphism $\rho_{\sigma}\colon A^{G}\rightarrow A^{G}$ by $\rho_{\sigma}(x)e:=w_{\sigma}xw_{\sigma}^{*}$ for all $x\in A^{G}$.
For each endomorphism $\rho$ of $A^{G}$,
the space 
\[\cH_{\rho}:=\{T\in M(A)\mid Tx=\rho(x)T
\mathrm{\ for\ all\ } x\in A^{G}\}
\]
is closed under the action of $G$ and which admits an inner product $\langle T| S\rangle_{\rho}:=S^{*}T$.
We write $(\cH_{\sigma},\langle\:|\:\rangle_{\sigma})$ for $(\cH_{\rho_\sigma}, \langle\:|\:\rangle_{\rho_\sigma})$.
Take matrix units $\{f_{\sigma, i, j}\}_{i, j}$ of $A_{\sigma}$ with $f_{\sigma, 1, 1}=p_\sigma$.
Since the partial isometries $\{w_{\sigma, i}:=w_{\sigma}f_{\sigma, 1, i}\}_ {i=1,2,\dots, \text{dim}(\sigma)}$ are contained in $\cM_{\hat{E}}$,
we have $w_{\sigma, i}=e\hat{E}(w_{\sigma, i})$ by (2) of Lemma \ref{lem dual}.
Define $v_{\sigma, i}:=\sqrt{\text{dim}(\sigma)}^{-1}\hat{E}(w_{\sigma, i})\in\cH_{\sigma}$ for all $\sigma\in\widehat{G}$ and $i=1,\dots, \text{dim}(\sigma)$.
Using the above notations,
we get the following as in Section 3 of \cite{ILP}.

\begin{lem}\label{endom}
The following hold. 
\begin{itemize}
\item[(1)] For every $\sigma, \pi\in\widehat{G}$ with $\sigma\neq\pi$,
$\rho_{\sigma}$ is an irreducible endomorphism and $(\rho_{\sigma}, \rho_{\pi})=0$.

\item[(2)] For every $\sigma\in\widehat{G}$,
we have $A_{\sigma}=\cH_{\sigma}^{*}e\cH_{\sigma}$.
\item[(3)] For every $\sigma\in\widehat{G}$,
the family $\{v_{\sigma, i}\}_{i}$ of isometries forms an orthonormal basis of $\cH_{\sigma}$ satisfying the Cuntz reration.
\item[(4)] There exists an isomorphism $(\cH_{\sigma}, \alpha)\cong(V_{\sigma}, \sigma)$ of $G$-Hilbert spaces for every $\sigma\in\widehat{G}$.
\item[(5)] For every $\sigma\in\widehat{G}$,
$\rho_{\sigma}$ is of finite index.
\item[(6)] For every $\sigma\in\widehat{G}$,
we have $\overline{\rho}_{\sigma}\sim\rho_{\overline{\sigma}}$ and there is an isometry $\overline{R}_{\sigma}\in(\id_{A^{G}}, \rho_{\sigma}\circ\rho_{\overline{\sigma}})$ such that $\cH_{\sigma}^{*}\overline{R}_{\sigma}=\cH_{\overline{\sigma}}$.
\end{itemize}
\end{lem}

\begin{proof}
For (1),
we show the irreducibility of $\rho_{\sigma}$ as follows:
\begin{align*}
M(A^{G})\cap\rho_{\sigma}(A^{G})^{\prime}
&\cong (M(A^{G})\cap\rho_{\sigma}(A^{G})^{\prime})e\\
&=w_{\sigma}w_{\sigma}^{*}M(A^{G})w_{\sigma}w_{\sigma}^{*}\cap(w_{\sigma}A^{G}w_{\sigma}^{*})^{\prime}\\
&\cong w_{\sigma}^{*}M(A^{G})w_{\sigma}\cap (A^{G}p_{\sigma})^{\prime}\\
&\subset p_{\sigma}(M(A\rtimes G)\cap (A^{G})^{\prime}) p_{\sigma}=\IC p_{\sigma}.
\end{align*}
To show $(\rho_{\sigma}, \rho_{\pi})=0$,
take $T\in(\rho_{\sigma}, \rho_{\pi})(\subset M(A^G))$.
By definition,
we have $w_{\pi}^{*}Tw_{\sigma}\in p_{\pi}(M(A\rtimes G)\cap (A^{G})^{\prime})p_{\sigma}=0$.
This implies $eT=eTe=0$.
Then $T=0$.

We show (2).
Since we have $\text{dim}(\sigma)v_{\sigma, i}^{*}ev_{\sigma, j}=f_{\sigma, i, j}$,
the inclusion $A_{\sigma}\subset\cH_{\sigma}^{*}e\cH_{\sigma}$ holds for every $\sigma\in\widehat{G}$.
By (1),
for all $\sigma\neq\pi\in\widehat{G}$,
we have $\cH_{\sigma}^{*}e\cH_{\sigma}\cH_{\pi}^{*}e\cH_{\pi}=\cH_{\sigma}^{*}e(\rho_{\pi}, \rho_{\sigma})\cH_{\pi}=0$.
Then,
each $\cH_{\sigma}^{*}e\cH_{\sigma}$ is contained in $M(A\rtimes G)\cap (A^{G})^{\prime}\cap (\prod_{\pi\neq\sigma}A_{\pi})^{\perp}$.
Since we have $\prod_{\pi\in\widehat{G}}A_{\pi}=M(A\rtimes G)\cap (A^{G})^{\prime}$,
$A_{\sigma}=\cH_{\sigma}^{*}e\cH_{\sigma}$ holds for every $\sigma\in\widehat{G}$.

Since $\{f_{\sigma, i, j}\}_{i,j}$ are matrix units of $A_{\sigma}\cong B(V_{\sigma})$,
there is an orthonormal basis $\{\xi_{i}\}_{i=1}^{\text{dim}(\sigma)}$ of $V_{\sigma}$ such that for all $i$ and $j$,
$f_{\sigma, i, j}=\text{dim}(\sigma)\int_{G}\langle\sigma(g)\xi_{j},\xi_{i}\rangle\lambda_{g}dg$ hold. 
Hence,
we have $v_{\sigma, i}^{*}v_{\sigma, j}=\hat{E}(v_{\sigma, i}^{*}ev_{\sigma, j})=\text{dim}(\sigma)^{-1}\hat{E}(f_{\sigma, i, j})=\delta_{i, j}$.
Since the conditional expectation $E$ is faithful and each $\rho_{\sigma}$ is irreducible,
a linear map $\cH_{\sigma}\ni u\mapsto v_{\sigma, 1}^{*}eu\in p_{\sigma}A_{\sigma}$ is injective. 
Then,
$\text{dim}\cH_{\sigma}\leq\text{dim}(p_{\sigma}A_{\sigma})=\text{dim}(\sigma)$ holds.
This implies that $\{v_{\sigma, i}\}_{i}$ consists orthogonal basis of $\cH_{\sigma}$.
To show the Cuntz relations,
it suffices to prove $\sum_{i=1}^{\rm{dim}(\sigma)}v_{\sigma, i}v_{\sigma, i}^{*}=1$.
Since the conditional expectation $E$ is faithful and $\sum_{i=1}^{\rm{dim}(\sigma)}v_{\sigma, i}v_{\sigma, i}^{*}\leq1$,
the following computation confirms $\sum_{i=1}^{\rm{dim}(\sigma)}v_{\sigma, i}v_{\sigma, i}^{*}=1$:
\begin{align*}
E\big(\sum_{i=1}^{\rm{dim}(\sigma)}v_{\sigma, i}v_{\sigma, i}^{*}\big)
=\hat{E}\big(e(\sum_{i=1}^{\rm{dim}(\sigma)}v_{\sigma, i}v_{\sigma, i}^{*})e\big)
=\frac{1}{\rm{dim}(\sigma)}\hat{E}\big(\sum_{i=1}^{\rm{dim}(\sigma)}e\big)
=1.
\end{align*}

To show (4),
we use the orthonormal basis $\{\xi_{i}\}_{i=1}^{\text{dim}(\sigma)}$ of $V_{\sigma}$ as above.
We have 
\begin{align*}
v_{\sigma, i}^{*}\alpha_{g}(v_{\sigma, j})&=\hat{E}(v_{\sigma, i}^{*}e\alpha_{g}(v_{\sigma, j}))\\
&=\hat{E}(v_{\sigma, i}^{*}ev_{\sigma, j}\lambda_{g^{-1}})\\
&=\text{dim}(\sigma)^{-1}\hat{E}(f_{\sigma, i, j}\lambda_{g^{-1}})\\
&=\langle\sigma(g)\xi_{j}, \xi_{i}\rangle,
\end{align*}
for every $g\in G$.
Therefore,
we get $(4)$.

For (5),
we define the conditional expectation $E_{\sigma}$ from $A^{G}$ onto $\rho_{\sigma}(A^{G})$ as 
$E_{\sigma}(x)=\rho_{\sigma}(E(v_{\sigma, 1}^{*}xv_{\sigma, 1}))=\text{dim}(\sigma)E(v_{\sigma, 1}E(v_{\sigma, 1}^{*}xv_{\sigma, 1})v_{\sigma, 1}^{*})$ for every $x\in A^G$.
Since $\alpha_{g}(v_{\sigma, 1})=\sum_{i}v_{\sigma, i}v_{\sigma, i}^{*}\alpha_{g}(v_{\sigma, 1})=\sum_{i}\langle\sigma(g)\xi_{1}, \xi_{i}\rangle v_{\sigma, i}$ and 
$\int_{G}\langle\sigma(g)\xi_{1}, \xi_{i}\rangle\overline{\langle\sigma(g)\xi_{1}, \xi_{j}\rangle}dg=\frac{\delta_{i, j}}{\text{dim}(\sigma)}$,
we have
\begin{align*}
E_{\sigma}(x)
&=\text{dim}(\sigma)E(v_{\sigma, 1}E(v_{\sigma, 1}^{*}xv_{\sigma, 1})v_{\sigma, 1}^{*})\\
&=\text{dim}(\sigma)\int_{G}\int_{G}\alpha_{g}(v_{\sigma, 1})\alpha_{h}(v_{\sigma, 1}^{*})x\alpha_{h}(v_{\sigma, 1})\alpha_{g}(v_{\sigma, 1}^{*})dgdh\\
&=\frac{1}{\text{dim}(\sigma)}\sum_{i, j}v_{\sigma, i}v_{\sigma, j}^{*}xv_{\sigma, j}v_{\sigma, i}^{*}
\end{align*}
for every $x\in A^{G}$.
Let $e_{i, j}:=v_{\sigma, i}v_{\sigma, j}^{*}$.
The conditional expectation $E_{\sigma}$ extends to a completely positive projection $\tilde{E}_{\sigma}$ of $A$ such that $\tilde{E}_{\sigma}(x)=\frac{1}{\text{dim}(\sigma)}\sum_{i, j}e_{i, j}xe_{j, i}$,
and this has a quasi basis $\{(\sqrt{\text{dim}(\sigma)}e_{i, j}, \sqrt{\text{dim}(\sigma)}e_{j, i})\}_{1\leq i, j\leq\text{dim}(\sigma)}$.
As stated after Definition \ref{quasi basis} and in Theorem \ref{p=e},
we have
\[
\mathrm{Ind}_{\mathrm{p}}\:\tilde{E}_{\sigma}=\mathrm{Ind}_{\mathrm{w}}\:\tilde{E}_{\sigma}={\rm dim}(\sigma)\sum_{i,j}e_{i,j}e_{j,i}={\rm dim}(\sigma)^{2}.
\]
Then,
the restriction $E_{\sigma}$ of $\tilde{E}_{\sigma}$ also has a finite index.

To show (6),
it suffices to show that $\overline{\rho}_{\sigma}\sim\rho_{\overline{\sigma}}$ by Lemma \ref{conjugate}.
We construct a $G$-equivariant unitary isomorphism between $(\cH_{\overline{\rho}_{\sigma}},\alpha)$ and $(V_{\overline{\sigma}}, \overline{\sigma})$. 
Let $R_{\rho_\sigma}$ and $\overline{R}_{\rho_\sigma}$ be isometries in Lemma \ref{conjugate}.
We set $u_i:=\sqrt{\text{dim}(\sigma)} v^{*}_{\sigma, i}\overline{R}_{\rho_\sigma}$,
then $\cH_{\overline{\rho}_{\sigma}}=\text{span}\{u_{i}\}_{i=1}^{\text{dim}(\sigma)}$ holds.
For every $g\in G$,
we have
\begin{align*}
u_{i}^{*}\alpha_{g}(u_{j})&=E(u_{i}^{*}\alpha_{g}(u_{j}))\\
&=\text{dim}(\sigma)\overline{R}_{\rho_\sigma}^{*}E(v_{\sigma, i}\alpha_{g}(v_{\sigma, j}^*))\overline{R}_{\rho_\sigma}\\
&=\text{dim}(\sigma)\sum_{k=1}^{\text{dim}(\sigma)}\overline{\langle\sigma(g)\xi_{j}, \xi_{k}\rangle}\overline{R}_{\rho_\sigma}^{*}E(v_{\sigma, i}v_{\sigma, k}^*)\overline{R}_{\rho_\sigma}\\
&=\sum_{k=1}^{\text{dim}(\sigma)}\overline{\langle\sigma(g)\xi_{j}, \xi_{k}\rangle}\overline{R}_{\rho_\sigma}^{*}\hat{E}(w_{\sigma, i}w_{\sigma, k}^*)\overline{R}_{\rho_\sigma}\\
&=\overline{\langle\sigma(g)\xi_{j}, \xi_{i}\rangle}\\
&=\langle\overline{\sigma}(g)\overline{\xi}_{j}, \overline{\xi}_{i}\rangle.
\end{align*}
Hence,
the linear map $\cH_{\overline{\rho}_{\sigma}}\ni u_{i}\mapsto\overline{\xi}_i\in V_{\overline{\sigma}}$ is a $G$-equivariant unitary isomorphism. 
The equivalence relation $\overline{\rho}_{\sigma}\sim\rho_{\overline{\sigma}}$ follows from the next lemma.
\end{proof}

We get the following as in Section 3 of [ILP].
\begin{lem}\label{isom}
Let $\cS_{0}:=\{[\rho]\in \mathrm{Sect}(A^{G})\mid \cH_{\rho}\neq0,\:\rho\text{ is irreducible.}\}$ and $\cS$ be the set of of all finite direct sums of $\cS_{0}$ in $\mathrm{Sect}(A^{G})$.
Then,
we have a natural bijection $\theta: \cS_{0}\rightarrow \widehat{G}$ defined by $\theta([\rho]):=[(\cH_{\rho}, \alpha)]$. 
Moreover,
$\theta$ extends to a bijection between $\cS$ and $\mathrm{Rep}_{\mathrm{f}}(G)$ which preserves direct sums and products.
\end{lem}

\begin{proof}
First,
we show that $\theta$ is a bijection between $\cS_{0}$ and $\widehat{G}$.
Thanks to (4) of Lemma \ref{endom},
there is an irreducible sector $[\rho_{\sigma}]\in\cS_{0}$ such that $\theta([\rho_{\sigma}])=\sigma$ for any $\sigma\in\widehat{G}$.
Conversely,
suppose $\rho$ is an irreducible endomorphism of $A^{G}$ with $\cH_{\rho}\neq0$.
We have $1_{M(A\rtimes G)}=\sum_{\sigma, i}v_{\sigma, i}^{*}ev_{\sigma, i}$, 
where the right-hand side converges in the strict topology.
Then,
for any $w\in\cH_{\rho}\setminus\{0\}(\subset M(A))$ and any $a\in A^{G}$,
we get $w^{*}ea=\sum_{\sigma, i}v_{\sigma, i}^{*}E(v_{\sigma, i}w^{*})ea$ and $E(v_{\sigma, i}w^{*})\in(\rho, \rho_{\sigma})$.
Hence,
there is at least one $\sigma\in\widehat{G}$ such that $(\rho, \rho_{\sigma})\neq0$.
By the irreducibility of $\rho$ and (1) of Lemma \ref{endom},
we can take a unique $\sigma$ and a unitary $u\in M(A^G)$ such that $(\rho, \rho_{\sigma})=\IC u$.
Therefore,
we have $\cH_{\rho}=u^{*}\cH_{\sigma}$ and $\theta([\rho])=\theta([\rho_{\sigma}])$.
This means $\theta$ is a well-defined bijection between $\cS_0$ and $\widehat{G}$.

Next,
we consider the natural extension $\tilde{\theta}\colon\cS\ni[\rho]\mapsto[(\cH_{\rho}, \alpha)] \in \mathrm{Rep}_{\mathrm{f}}(G)$ of $\theta$.
By the uniqueness of irreducible decompositions of finite index endomorphisms
(see Lemma 4.1 of \cite{I1}),
$\tilde{\theta}$ is bijective.
The rest of the statement follows from straightforward calculations.
\end{proof}
\begin{rmk}
\label{rmk isom}
Using this lemma,
for every $\sigma\in\mathrm{Rep}_\mathrm{f}(G)$,
we can take a unique endomorphism $\rho_{\sigma}\in\mathrm{End}(A^G)$,
up to equivalence,
such that $\tilde{\theta}(\rho_{\sigma})=\sigma$.
\end{rmk}
\begin{rmk}\label{rmk quasi product}
   For every $\sigma\in\widehat{G}$,
   $\rho_{\sigma}$ can be identified with $\check{\alpha}_{\sigma}$ in Section 2.4 of \cite{I2}.
   When $A$ is separable,
   by Lemma 2.14 of \cite{I2},
   $\alpha$ is quasi-product if and only if every $\rho_{\sigma}$ is properly outer.
\end{rmk}
In Section 4,
we discuss \Cs-irreducibility of the inclusion $A^{G}\subset A$ to show the Galois correspondence.
\Cs-irreducibility of inclusions arising from discrete crossed products is discussed in Theorem 5.8 of \cite{R}.
In this section,
Lemma \ref{decomp} shows that $A$ admits a crossed product decomposition as in Section 6 of \cite{I1}. 
Before presenting the last lemma of this section, we recall the definition of \Cs-irreducible inclusions.
\begin{defn}[See Definition 3.1 of \cite{R} for unital inclusions]
A nondegenerate inclusion $B\subset A$ of \Cs-algebras is said to be \Cs-irreducible if every intermediate \Cs-subalgebra is simple.
\end{defn}
\begin{rmk}\label{rmk irrd}
When the inclusion $B\subset A$ is \Cs-irreducible,
we have $M(A)\cap B^{\prime}=\IC$.
(See Remark 3.8 of \cite{R} for unital inclusions.)
If $M(A)\cap B^{\prime}\neq\IC$,
then there are nonzero positive elements $x,y$ in $M(A)\cap B^{\prime}$ with $xy=yx=0$.
The intermediate \Cs-subalgebra $C:=\overline{B+xAx+yAy}$ contains a non-trivial ideal $\overline{xAx}$.
\end{rmk}

The following lemma was pointed out by Izumi
(see Lemma 2.15 of \cite{I2}).
For the reader's convenience, we give details of the proof.
\begin{lem}[Lemma 2.15 of \cite{I2}]\label{decomp}
Let $A_{0}:=\mathrm{span}\{wx\mid w\in \cH_{\sigma}^{*},\; x\in A^{G},\; \sigma\in\widehat{G}\}$.
Then,
$A_{0}$ is a dense $*$-subalgebra of $A$,
which is closed under the action of $G$.
\end{lem}

\begin{proof}
First,
we show that $A_0$ is a $G$-invariant $*$-subalgebra of $A$.
Since $\alpha_{g}(\cH_{\sigma})=\cH_{\sigma}$ for any $g\in G$ and $\sigma\in\widehat{G}$,
$A_{0}$ is $G$-invariant.
For any $\sigma\in\widehat{G}$ and $i$,
the inclusion
\[
(v_{\sigma, i}^{*}A^{G})^{*}=A^{G}v_{\sigma, i}=A^{G}v_{\overline{\sigma}, i}^{*}R_{\sigma}=v_{\overline{\sigma}, i}^{*}\rho_{\overline{\sigma}}(A^{G})R_{\sigma}\subset A_{0}
\]
holds.
Then $A_{0}^{*}=A_{0}$.
By Lemma \ref{isom},
we have $(\iota, \iota\circ\rho_{\pi}\circ\rho_{\sigma})=\cH_{\pi}\cH_{\sigma}\cong\cH_{\pi}\otimes\cH_{\sigma}$ and there are isometries $\{S_{i}\}_{i=1}^{n}$ in $M(A^{G})$ with the Cuntz relation such that $\cH_{\pi}\cH_{\sigma}=\sum_{i}S_{i}\cH_{\sigma_{i}}$,
where $\sigma_{i}\in\widehat{G}$ are not necessarily pairwise orthogonal and $\sigma\otimes\pi$ is decomposed into $\oplus_{i}\sigma_{i}$.
Thus,
we have
\[
\cH_{\sigma}^{*}A^{G}\cH_{\pi}^{*}A^{G}=\cH_{\sigma}^{*}\cH_{\pi}^{*}\rho_{\pi}(A^{G})A^{G}=\sum_{i}\cH_{\sigma_{i}}^{*}S_{i}^{*}A^{G}\subset A_{0}
\]
for every $\sigma, \pi\in\widehat{G}$.
Consequently,
$A_{0}$ is a $G$-invariant $*$-subalgebra of $A$.

Let $\chi_{\sigma}$ be the character of $\sigma\in \widehat{G}$ and $P_{\sigma}\colon A\rightarrow A$ be a linear map defined by $P_{\sigma}(x):=\int_{G}\chi_{\sigma}(g)\alpha_{g}(x)dg$ for $x\in A$.
Since we have 
\begin{align*}
    \sum_{i=1}^{{\rm dim}(\sigma)}v_{\sigma,i}^{*}E(v_{\sigma,i}x)
    =&\sum_{i=1}^{{\rm dim}(\sigma)}v_{\sigma,i}^{*}\int_{G}\alpha_{g}(v_{\sigma,i}x)dg\\
    =&\sum_{i=1}^{{\rm dim}(\sigma)}v_{\sigma,i}^{*}\int_{G}\sum_{j=1}^{{\rm dim}(\sigma)}v_{\sigma,j}(v_{\sigma,j}^{*}\alpha_{g}(v_{\sigma,i}))\alpha_{g}(x)dg\\
    =&\int_{G}\big(\sum_{i=1}^{{\rm dim}(\sigma)}v_{\sigma,i}^{*}\alpha_{g}(v_{\sigma,i})\big)\alpha_{g}(x)dg
    =P_{\sigma}(x),
\end{align*}
$P_{\sigma}(x)$ is contained in $A_{0}$ for all $x\in A$ and $\sigma\in\widehat{G}$.
If $A_{0}$ is not dense in $A$,
then there exists a non-zero linear functional $\varphi$ in $A^{*}$ with $\varphi(A_{0})=0$.
For every $x\in A$ and $\sigma\in\widehat{G}$,
\[
0=\varphi(P_{\sigma}(x))=\int_{G}\chi_{\sigma}(g)\varphi(\alpha_{g}(x))dg
\]
holds.
By the Peter-Weyl theorem,
this implies $\varphi(\alpha_{g}(x))=0$ for every $g\in G$ and every $x\in A$,
which contradicts the assumption that $\varphi$ is non-zero.
Then,
$A_{0}$ is dense in $A$.
\end{proof}

\section{Inclusions of simple \Cs-algebras with conditional expectations}
As discussed in the previous section,
under some assumptions,
an inclusion $A^{G}\subset A$ generated by a compact group action $\alpha\colon G\acts A$ has a similar property to those of inclusions $B\subset B\rtimes_{r}\Gamma$ generated by discrete group actions.
Hern{\'a}ndez Palomares and Nelson showed that if a unital inclusion $A\subset B=A\rtimes \IB$ is generated by an outer action of a unital tensor category $\cC$ and a $\cC$-graded \Cs-algebra $\IB$,
then the lattice $\{D\mid A\subset D\subset B,\;
A\subset D\in\textbf{\Cs-disc}\}$
of intermediate discrete inclusions is isomorphic to the lattice of $\cC$-graded \Cs-subagebras of $\IB$.
(See Theorem G of \cite{NP} for details.)
When a discrete group action $\Gamma\acts A$ on a unital simple \Cs-algebra is outer,
for every intermediate \Cs-algebra $D$ between $A$ and $A\rtimes_{r}\Gamma$,
the inclusion $A\subset D$ is automatically discrete.
(See Corollary 5.14 of \cite{NP}.)
In this section,
we consider inclusions such that the lattices of all intermediate subalgebras have rigid structures.
For example, 
Section 3 of \cite{R} discusses the relative Diximer property and the relative excision property for inclusions of \Cs-algebras.
We use the following notation.

\begin{no}
Let $B\subset A$ be a nondegenerate inclusion of \Cs-algebras. 
We say that $B\subset A$ satisfies Condition ($*$) if there is a positive element $b_{0}\in B$ with $\|b_0\|=1$ which satisfies the following.
\begin{itemize}
\item For any element $x\in A$ and $\epsilon>0$,
there exist finitely many elements $h_{1}, \dots, h_{n}\in B$ and $y\in B\text{  s.t. }$
\begin{equation}\label{ridgid}
\|\sum_{i}h_{i}^{*}h_{i}\|\leq1,
\;\|\sum_{i}h_{i}^{*}b_{0}h_{i}\|\approx_{\epsilon}1,
\text{ and } \sum_{i}h_{i}^{*}xh_{i}\approx_{\epsilon}y.
\end{equation}
\end{itemize}
\end{no}
First,
we check the following simple properties of Condition ($*$).

\begin{lem}\label{easy}
Suppose a nondegenerate inclusion $B\subset A$ satisfies Condition ($*$), 
then the following hold:
\begin{itemize}
\item[(1)] If $C$ is an intermediate \Cs-algebra between $B$ and $A$, 
then the inclusions $B\subset C$ and $C\subset A$ also satisfy Condition ($*$).

\item[(2)] The inclusion $B\otimes\IK\subset A\otimes\IK$ also satisfies Condition ($*$).

\item[(3)] If $B$ is simple,
then for any finitely many elements $x_{1}, x_{2},\dots, x_{m}\in A$,
any positive element $b\in B$,
and any $\epsilon>0$,
there exist finitely many elements $h_{1}, \dots, h_{n}\in B$ and $y_{1},\dots, y_{m}\in B$ satisfying
\[
\|\sum_{i}h_{i}^{*}h_{i}\|\leq1,
\;\sum_{i}h_{i}^{*}bh_{i}\approx_{\epsilon}b,\text{ and }
\sum_{i}h_{i}^{*}x_{j}h_{i}\approx_{\epsilon}y_{j}
\] for all $j$. 

\item[(4)] If there is a conditional expectation $E$ from $A$ onto $B$ and $B$ is simple,
then for any positive element $b\in B$,
finitely many elements $x_1,\dots, x_m\in A$,
and any $\epsilon>0$,
there exist finitely many elements $h_{1}, \dots, h_{n}\in B$ such that
\[
\|\sum_{i}h_{i}^{*}h_{i}\|\leq1,
\;\sum_{i}h_{i}^{*}bh_{i}\approx_{\epsilon}b, 
\text{ and } \sum_{i}h_{i}^{*}x_{j}h_{i}\approx_{\epsilon}E(\sum_{i}h_{i}^{*}x_{j}h_{i})
\]
hold for all $j$.

\item[(5)] If $B$ is simple and there exists a faithful conditional expectation $E$ from $A$ onto $B$,
then $A$ is also simple.
Moreover,
the inclusion $B\subset A$ is \Cs-irreducibe.
\end{itemize}
\end{lem}
\begin{proof}
(1) and (2) are obtained by straightforward discussion.

To show (3),
let $b_{0}\in B$ be a positive element in the definition of Condition ($*$).
First,
we assume $m=1$ and fix elements $b\in B_{+}$ and $x_{1}\in A$.
Since the inclusion $B\subset A$ satisfies Condition ($*$),
the statement (3) is trivial for $b=0$ and $m=1$.
We may assume $\|b\|=1$.
Since $B$ is simple,
by Lemma A.2 of \cite{S},
for any $\epsilon>0$,
we obtain finitely many elements $s_{1}, s_{2},\dots,s_{N}$ in $B$
such that $\sum_{i=1}^{N}s_{i}^{*}bs_{i}\approx_{\frac{\epsilon}{3}}b_{0}$ and $\|\sum_{i=1}^{N}s_{i}^{*}s_{i}\|\leq1$ hold.
Let $x:=\sum_{i=1}^{N}s_{i}^{*}x_{1}s_{i}$.
The assumption (\ref{ridgid}) implies that there exist elements $h_{1},h_{2},\dots,h_{n}$ and $y$ in $B$ satisfying
\[\|\sum_{i,j}(s_{i}h_{j})^{*}(s_{i}h_{j})\|\leq1,\; \|\sum_{i,j}(s_{i}h_{j})^{*}b(s_{i}h_{j})\|\approx_{\frac{2\epsilon}{3}}1,\; {\rm and}\;
\sum_{i,j}(s_{i}h_{j})^{*}x_{1}(s_{i}h_{j})\approx_{\frac{\epsilon}{3}}y.
\]
Using Lemma A.2 of \cite{S} again,
we get elements $t_{1},t_{2},\dots,t_{M}$ in $B$ satisfying
\[
\|\sum_{i,j,k}(s_{i}h_{j}t_{k})^{*}(s_{i}h_{j}t_{k})\|\leq1,\hspace{15pt} \sum_{i,j,k}(s_{i}h_{j}t_{k})^{*}b(s_{i}h_{j}t_{k})\approx_{\epsilon}b,\]
and
\[
\sum_{i,j,k}(s_{i}h_{j}t_{k})^{*}x_{1}(s_{i}h_{j}t_{k})\approx_{\epsilon}\sum_{k}t_{k}^{*}yt_{k}.
\]
Set $y_{1}:=\sum_{k}t_{k}^{*}yt_{k}$, then we get the statement for $m=1$.
Suppose (3) holds for $m=m_{0}-1$ and take elements $x_{1},x_{2},\dots,x_{m_0}$ in $A$ and $\epsilon>0$.
We can choose $h_{1},\dots,h_{n}$ and $y_{1}^{\prime},\dots,y_{m_{0}-1}^{\prime}$ in $B$ such that
\[
\|\sum_{i}h_{i}^{*}h_{i}\|\leq1,
\;\sum_{i}h_{i}^{*}bh_{i}\approx_{\epsilon}b,\textrm{\ and\ }
\sum_{i}h_{i}^{*}x_{j}h_{i}\approx_{\epsilon}y_{j}^{\prime}
\] hold for all $1\leq j\leq m_{0}-1$.
We also get $h_{1}^{\prime},\dots,h_{l}^{\prime}$ and $y_{m_{0}}$ in $B$ with
\[
\|\sum_{k}h_{k}^{\prime*}h_{k}^{\prime}\|\leq1,
\;\sum_{k}h_{k}^{\prime*}bh_{k}^{\prime}\approx_{\epsilon}b,\text{ and }
\sum_{k}h_{k}^{\prime*}(\sum_{i}h_{i}^{*}x_{m_{0}}h_{i})h_{k}^{\prime}\approx_{\epsilon}y_{m_0}.
\]
Let $y_{j}:=\sum_{k=1}^{l}h_{k}^{\prime*}y_{j}^{\prime}h_{k}^{\prime}$ for $1\leq j\leq m_{0}-1$.
Then,
we have
\[
\|\sum_{i,k}h_{k}^{\prime*}h_{i}^{*}h_{i}h_{k}^{\prime}\|\leq1,
\;\sum_{i,k}h_{k}^{\prime*}h_{i}^{*}bh_{i}h_{k}^{\prime}\approx_{2\epsilon}b,\text{ and }
\sum_{i,k}h_{k}^{\prime*}h_{i}^{*}x_{j}h_{i}h_{k}^{\prime}\approx_{2\epsilon}y_{j}
\] for $1\leq j\leq m_{0}$.
By induction,
we get (3).

(4) is clear from (3).

To prove (5),
take a nonzero positive element $x_{0}$ in $A$.
To show \Cs-irreducibility of the inclusion $B\subset A$,
it is enough to show that the closed $B$-bimodule generated by $x_{0}$ contains approximate units $(d_{\nu})_{\nu}$ of $B$.
We note that the net $(d_{\nu})_{\nu}$ is approximate units of every intermediate subalgebra between $B$ and $A$ since the inclusion is nondegenerate.
We may assume $\|E(x_{0})\|=1$.
Using Lemma A.2 of \cite{S},
for any $\epsilon>0$ and any $d_{\nu}$,
we can take $s_{1},s_{2},\dots,s_{m}$ in $B$ such that $\|\sum_{j=1}^{m}s_{j}^{*}s_{j}\|\leq1$ and $\sum_{j=1}^{m}s_{j}^{*}E(x_{0})s_{j}\approx_{\epsilon}d_{\nu}$ hold.
Applying (4) with $x:=\sum_{j=1}^{m}s_{j}^{*}x_{0}s_{j}$ and $b:=\sum_{j=1}^{m}s_{j}^{*}E(x_{0})s_{j}$,
we obtain elements $h_{1},\dots,h_{n}$ in $B$ satisfying
\[
\|\sum_{i,j}(s_{j}h_{i})^{*}(s_{j}h_{i})\|\leq1,
\;\sum_{i,j}(s_{j}h_{i})^{*}E(x_{0})(s_{j}h_{i})\approx_{\epsilon}\sum_{j}s_{j}^{*}E(x_{0})s_{j}\approx_{\epsilon}d_{\nu}
\]
and
\[
\sum_{i,j}(s_{j}h_{i})^{*}x_{0}(s_{j}h_{i})\approx_{\epsilon}
E\big(\sum_{i,j}(s_{j}h_{i})^{*}x_{0}(s_{j}h_{i})\big)
=\sum_{i,j}(s_{j}h_{i})^{*}E(x_{0})(s_{j}h_{i}).
\]
The above approximations imply $\sum_{i,j}(s_{j}h_{i})^{*}x_{0}(s_{j}h_{i})\approx_{3\epsilon}d_{\nu}$. 
Hence,
the closed $B$-bimodule generated by $x_{0}$ contains $d_{\nu}$,
and we get (5).
\end{proof}

We give examples of inclusions with Condition ($*$)  arising from compact group actions.

\begin{exa}\label{exa ab}
Let $\alpha\colon G\curvearrowright A$ be a quasi-product action of a compact second countable group $G$ on a separable simple \Cs-algebra.
If the fixed point algebra $A^G$ is simple, 
then the inclusion $A^{G}\subset A$ satisfies Condition ($*$).
Moreover,
the inclusion $A^{G}\subset A$ is \Cs-irreducible.
First,
we assume that $\alpha$ is stable.
To show Condition ($*$),
fix a positive element $b_{0}\in A^G$ with $\|b_{0}\|=1$. 
Since $\alpha$ is a quasi-product action,
we have $M(A)\cap (A^{G})^{\prime}=\IC$ 
(see Remark \ref{rmk prime}).
By Lemma \ref{decomp},
for any $x\in A$ and $\epsilon>0$,
there is a finite set $F\subset\widehat{G}$ such that $x\approx_{\epsilon}\sum_{\sigma\in F, i}v_{\sigma, i}^{*}a_{\sigma, i}$,
where $a_{\sigma, i}\in A^G$ for all $\sigma\in F$ and all $i$.
As in Remark \ref{rmk quasi product},
endomorphisms $\{\rho_{\sigma}\}_{\sigma\in F\setminus\{1_{\widehat{G}}\}}$ are properly outer.
By Lemma \ref{lem prop outer},
there is a positive element $c\in A^{G}$ with $\|c\|=1$ such that $\|cb_{0}c\|\approx_{\epsilon}1$ and $\rho_{\sigma}(c)a_{\sigma, i}c\approx_{\frac{\epsilon}{M}}0$ hold for every $i$ and $\sigma\neq 1_{\widehat{G}}$,
where $M:=\sum_{\sigma\in F}{\rm dim}(\sigma)$.
Set $y:=ca_{1_{\widehat{G}}}c\in A^{G}$.
Then,
we get $cxc\approx_{2\epsilon}y$ and $\|cb_{0}c\|\approx_{\epsilon}1$.
For general $\alpha$,
the inclusion $A^{G}\otimes\IK\subset A\otimes\IK$ satisfies the conditon ($*$).
It is routine to show that the consequence of (3) of Lemma 4.1 is preserved by taking a corner subalgebra.
Hence,
the inclusion $A^{G}\subset A$ also satisfies Condition ($*$).

If a compact group $G$ is abelian or profinite,
and the stabilized action $\alpha\otimes{\rm id}_{\IK}\colon G\acts A\otimes\IK$ of $\alpha$ satisfies the assumptions in Section 3,
then every endomorphism $\rho_{\sigma}$ is an automorphism or has a finite depth.
Lemma 1.1 of \cite{Ki} and Theorem 7.5 of \cite{I1} imply that each $\rho_{\sigma}$ is properly outer when $\sigma\neq1_{\widehat{G}}$.  
By the same discussion as in the first half,
we can show that the inclusion $A^{G}\subset A$ satisfies Condition ($*$).
In this case, separability of $A$ is not required.
\end{exa}

Thanks to a very recent result \cite[Theorem 1.1]{I2} by Izumi,
when $A$ is separable and a stabilized action $\alpha\otimes{\rm id}_{\IK}\colon G\acts A\otimes\IK$ satisfies the assumptions in Section 3,
$\alpha$ is quasi-product and the inclusion $A^G\subset A$ satisfies Condition ($*$).
However,
since this theorem is based on a very deep result of \cite{BEK},
we then give examples of inclusions arising from compact group actions such that we can confirm Condition ($*$) in a more primitive way.
To explain the next example,
we use the following lemmas.

\begin{lem}\label{simple}
Let $B$ be a $\sigma$-unital simple \Cs-algebra.
For any finitely many elements $x_{1}, x_{2},\dots, x_{n}\in M(B)_{+}$ and any $\epsilon>0$,
there are countably many elements $(s_{i})_{i=1}^{\infty}$ in $B$ with $\sum_{i}s_{i}s_{i}^{*}=1_{M(B)}$ and nonnegative numbers $\mu_{1},\dots, \mu_{n}\geq0$ such that
\[
\mu_{1}=\|x_1\|,\;
\mu_{k}\leq\|x_{k}\|,
\text{ and }\sum_{i}s_{i}x_{k}s_{i}^{*}\approx_{\epsilon}\mu_{k}1_{M(B)}
\]
hold for $k=1,\dots,n$,
where the sums $\sum_{i}s_{i}s_{i}^{*}$ and $\sum_{i}s_{i}x_{k}s_{i}^{*}$ converge in the strict topology.
\end{lem}
\begin{proof}
Before the proof,
we remark that for any sequence $(t_{i})_{i=1}^{\infty}$ in $B$ and any element $y\in M(B)$,
if the sum $\sum_{i}t_{i}t_{i}^{*}$ converges in the strict topology,
then $\sum_{i}t_{i}yt_{i}^{*}$ also converges in the strict topology
(see Lemma 6.3 of \cite{OP}).

Fix $x_{1},\dots,x_{n}\in M(B)$.
For $\delta>0$,
take a positive function $f_{\delta}\in C([0, 1])$ that is supported by $[1-\delta, 1]$ and $f_{\delta}(t)=1$ for all $t\in[1-\frac{\delta}{2}, 1]$.
Then,
for every $y\in\overline{f_{\delta}(\frac{1}{\|x_{1}\|}x_1)Bf_{\delta}(\frac{1}{\|x_{1}\|}x_1)}$,
we have $x_{1}y\approx_{\|x_{1}\|\|y\|\delta}\|x_{1}\|y$.
Since the hereditary subalgebra $H:=\overline{f_{\delta}(\frac{1}{\|x_{1}\|}x_1)Bf_{\delta}(\frac{1}{\|x_{1}\|}x_1)}$ is full and $\sigma$-unital,
there is a sequence $(s_{i}^{(1)})_{i=1}^{\infty}\subset B$ such that the sum $\sum_{i}s_{i}^{(1)}s_{i}^{(1)*}$ converges to $1$ in the strict topology and $s_{i}^{(1)*}s_{i}^{(1)}\in H$ by Lemma 6.2 of \cite{OP}.
Since 
\[
x_{1}f_{2\delta}(\frac{1}{\|x_{1}\|}x_1)\approx_{2\|x_{1}\|\delta}\|x_{1}\|f_{2\delta}(\frac{1}{\|x_{1}\|}x_1)\]
and 
\[
s_{i}^{(1)}f_{2\delta}(\frac{1}{\|x_{1}\|}x_1)=s_{i}^{(1)}
\]
hold,
we get 
\[
\|x_1\|1=\|x_1\|\sum_{i}s_{i}^{(1)}s_{i}^{(1)*}\approx_{2\|x_{1}\|\delta}\sum_{i}s_{i}^{(1)}x_{1}s_{i}^{(1)*}.
\]
Similarly,
if we put $x_{k}^{(2)}:=\sum_{i}s_{i}^{(1)}x_{k}s_{i}^{(1)*}$ and $\mu_2:=\|x_{2}^{(2)}\|$, 
then we get a sequence $(s_{i}^{(2)})_{i=1}^{\infty}\subset B$ such that 
\[
\sum_{i}s_{i}^{(2)}s_{i}^{(2)*}=1\mathrm{\ and\ }\sum_{i}s_{i}^{(2)}x_{2}^{(2)}s_{i}^{(2)*}\approx_{2\mu_{2}\delta}\mu_{2}1
\] hold. 
Repeating this strategy for $k=1,\dots,n$,
we get sequences $\{(s_{i}^{(k)})_{i=1}^{\infty}\}_{k=1}^{n}$ in $B$ and nonnegative numbers $\mu_{k}\leq\|x_k\|$ for $k=1,\dots,n$ which satisfy
\[\sum_{i_{1},\dots,i_{n}}s_{i_{n}}^{(n)}\dots s_{i_2}^{(2)}s_{i_1}^{(1)}s_{i_1}^{(1)*}s_{i_2}^{(2)*}\dots s_{i_{n}}^{(n)*}=1\]
and
\[\sum_{i_{1},\dots,i_{n}}s_{i_{n}}^{(n)}\dots s_{i_2}^{(2)}s_{i_1}^{(1)}x_{k}s_{i_1}^{(1)*}s_{i_2}^{(2)*}\dots s_{i_{n}}^{(n)*}\approx_{2\mu_{k}\delta}\mu_{k}1.\]
If we choose $\delta$ as $2\|x_{k}\|\delta\leq\epsilon$ for all $k$,
then we get the statement.
\end{proof}

\begin{lem}\label{simple1}
Let $B$ be a $\sigma$-unital simple \Cs-algebra.
For any finitely many elements $x_{1}\in M(B)_{+}$ and $x_{2},\dots, x_{n}\in M(B)$ and any $\epsilon>0$,
there are countably many elements $(s_{i})_{i=1}^{\infty}$ in $B$ with $\sum_{i}s_{i}s_{i}^{*}=1_{M(B)}$ and complex numbers $\mu_{1},\dots, \mu_{n}$ such that 
\[
\mu_{1}=\|x_{1}\|,\;
|\mu_{k}|\leq4\|x_{k}\|,\;
and \sum_{i}s_{i}x_{k}s_{i}^{*}\approx_{\epsilon}\mu_{k}1_{M(B)}
\]
hold for $k=1,\dots,n$,
where the sums $\sum_{i}s_{i}s_{i}^{*}$ and $\sum_{i}s_{i}x_{k}s_{i}^{*}$ converge in the strict topology.
\end{lem}
\begin{proof}
Since each $x_k$ is decomposed into a linear span of four positive elements,
the statement follows from the previous lemma.
\end{proof}

\begin{exa}\label{exa isa1}
Let $A$ be a separable simple \Cs-algebra and $\alpha\colon G\acts A$ be an isometrically shift-absorbing action of a compact second countable group $G$.
Then $A^G$ is simple and the inclusion $A^{G}\subset A$ satisfies Condition ($*$).
Moreover,
the inclusion $A^{G}\subset A$ is \Cs-irreducible by Lemma \ref{easy} (5).
Let $\theta\colon L^{2}(G, A)\rightarrow A_{\infty, \alpha}$ be an equivariant $A$-bimodule map in Theorem \ref{char} (Proposition 3.8 of \cite{GS}).
For any $x\in A$,
let $\tilde{x}$ be an element of $L^{2}(G, A)^{G}$ such that $\tilde{x}(g):=\alpha_{g}(x)$.
Define the $A^{G}$-bimodule map $\psi\colon A\rightarrow A_{\infty, \alpha}^{G}$ as $\psi(x):=\theta(\tilde{x})$.
By definition,
we have
\[
\psi(x)^{*}\psi(y)=\langle\tilde{x}, \tilde{y}\rangle=\int_{G}\alpha_{g}(x^{*}y)dg=E(x^{*}y)
\]
for all $x, y\in A$.
Fix a positive element $b_{0}$ in $A^G$ with $\|b_{0}\|=1$,
and we show (\ref{ridgid}).
Take an element $x$ in $A$ and $\epsilon>0$. 
There is a measurable partition $G=\sqcup_{k=1}^{n}E_{k}$ and elements $g_{k}\in E_{k}$ such that $\|\alpha_{g_{k}^{-1}}(x)-\alpha_{g^{-1}}(x)\|<\epsilon$ for $k=1,\dots, n$ and all $g\in E_{k}$.
By Lemma \ref{simple},
there are countably many elements $(s_i)_{i=1}^{\infty}$ in $A$ and complex numbers $\mu_{1},\dots, \mu_{n}$ which satisfy $\sum_{i}s_{i}s_{i}^{*}=1$, $\sum_{i}s_{i}b_{0}s_{i}^{*}\approx_{\epsilon}1$,
and $\sum_{i}s_{i}\alpha_{g_{k}^{-1}}(x)s_{i}^{*}\approx_{\epsilon}\mu_{k}$.
Since the sum $\sum_{i}b_{0}^{\frac{1}{2}}\alpha_{g}(s_{i})x\alpha_{g}(s_{i})^{*}b_{0}^{\frac{1}{2}}$ converges uniformly on $G$,
we get the following approximations;
\begin{align*}
\sum_{i}b_{0}^{\frac{1}{2}}\psi(s_{i})x\psi(s_{i})^{*}b_{0}^{\frac{1}{2}}
=&\sum_{i}b_{0}^{\frac{1}{2}}\theta(\tilde{s}_{i})x\theta(\tilde{s}_{i})^{*}b_{0}^{\frac{1}{2}}\\
=&\sum_{i}\langle \tilde{s}^{*}_{i}b_{0}^{\frac{1}{2}},x\tilde{s}^{*}_{i}b_{0}^{\frac{1}{2}}\rangle_{L^{2}(G, A)}\\
=&\int_{G}\sum_{i}b_{0}^{\frac{1}{2}}\alpha_{g}(s_{i})x\alpha_{g}(s_{i})^{*}b_{0}^{\frac{1}{2}}dg\\
=&\int_{G}\alpha_{g}(\sum_{i}b_{0}^{\frac{1}{2}}s_{i}\alpha_{g^{-1}}(x)s_{i}^{*}b_{0}^{\frac{1}{2}})dg
\approx_{2\epsilon}\sum_{k=1}^{n}\mu_{k}|E_{k}|b_{0},
\end{align*}
where each $|E_{k}|$ is a measure of $E_{k}$.
Similarly,
we get $\sum_{i}b_{0}^{\frac{1}{2}}\psi(s_{i})b_{0}\psi(s_{i})^{*}b_{0}^{\frac{1}{2}}\approx_{\epsilon}b_{0}$ and $\sum_{i}b_{0}^{\frac{1}{2}}\psi(s_{i})\psi(s_{i})^{*}b_{0}^{\frac{1}{2}}=b_{0}$.
Put $y:=\sum_{k=1}^{n}\mu_{k}|E_{k}|b_{0}\in A^{G}$.
Since $\psi(s_{i})^{*}b_{0}^{\frac{1}{2}}$ is in $(A_{\infty, \alpha})^G=(A^{G})_{\infty}$,
there exist lifts $(h_{i, j})_{j=1}^{\infty}\in l^{\infty}(\IN, A^{G})$ of $\psi(s_{i})^{*}b_{0}^{\frac{1}{2}}$ such that
\begin{align*}
&\sum_{i=1}^{M}h_{i, l}^{*}h_{i, l}\leq b_{0},\\
\lim_{N\to\infty}\limsup_{j\to\infty}&\|\sum_{i=1}^{N}h_{i, j}^{*}xh_{i, j}-y\|<2\epsilon,\\
\lim_{N\to\infty}\limsup_{j\to\infty}&\|\sum_{i=1}^{N}h_{i, j}^{*}b_{0}h_{i, j}-b_{0}\|<\epsilon
\end{align*}
hold for every $M$ and $l$.
Thus,
we can choose  $j$ and $N$ such that $\|\sum_{i=1}^{N}h_{i, j}^{*}h_{i, j}\|\leq1$,
$\|\sum_{i=1}^{N}h_{i, j}^{*}xh_{i, j}-y\|<3\epsilon$,
and $\|\sum_{i=1}^{N}h_{i, j}^{*}b_{0}h_{i, j}-b_{0}\|<2\epsilon$ hold.
This implies that the inclusion $A^{G}\subset A$ satisfies Condition ($*$).

To show that $A^G$ is simple,
let $c, d$ be positive elements in $A^G$ with $\|c\|=\|d\|=1$.
Since $A$ is simple,
for any $\epsilon>0$,
we can take finitely many elements $t_{1},\dots,t_{m}$ in $A$ with $\sum_{i=1}^{m}t_{i}ct_{i}^{*}\approx_{\epsilon}d$.
Then,
we have 
\[
\sum_{i=1}^{m}\psi(t_{i})c\psi(t_{i})^{*}=E(\sum_{i=1}^{m}t_{i}ct_{i}^{*})\approx_{\epsilon}E(d)=d.
\]
Since $\psi(t_{i})\in(A^{G})_{\infty}$,
by the same arguments as in the first half,
we get $d\in\overline{\rm span}\;A^{G}cA^{G}$. 
This implies that $A^G$ is simple.
\end{exa}

\begin{rmk}\label{rmk_pis}
Under the above assumptions,
we assume $A$ is purely infinite simple.
For any nonzero positive elements $a, b\in A^G$,
there is an element $t\in A$ such that $tat^{*}=b$.
Hence,
we have $\psi(t)a\psi(t)^{*}=E(tat^{*})=b$.
This implies that there is a sequence $\{h_{n}\}_n$ of $A^G$ with $\lim_{n\to\infty}h_{n}ah_{n}^{*}=b$.
Therefore,
$A^G$ is also purely infinite simple
(see Proposition 4.1.1 of \cite{RS}).
\end{rmk}

Finally,
we give an example of an inclusion of free product \Cs-algebras.
\begin{exa}\label{exa free}
Let $A\subset \cA$ and $B\subset \cB$ be inclusions of unital \Cs-algebras with faithful conditional expectations $E_A$ and $E_B$,
respectively.
Suppose $\phi_A$ (resp. $\phi_B$) is a state of $A$ (resp. $B$) whose GNS representation is faithful.
Then,
there is a natural inclusion
\[
(A, \phi_{A})*(B, \phi_{B})\subset(\cA, \phi_{A}\circ E_A)*(\cB, \phi_{B}\circ E_B)
\]
of reduced free product \Cs-algebras.
This inclusion satisfies Condition ($*$) if there is a Haar unitary $u$ in the centralizer of $\phi_A$,
(i.e.,
$\phi_A(u^n)=0$ for every $n\in\IZ\setminus\{0\}$),
and if $B\neq\IC$.

Let $\fA$ be the free product \Cs-algebra $(\cA, \phi_{A}\circ E_A)*(\cB, \phi_{B}\circ E_B)$
and $\phi$ be the canonical state of $\fA$ associated with $\phi_{A}\circ E_A$ and $\phi_{B}\circ E_B$,
(see \cite[Definition 4.7.1]{BO} for the definition).
By Proposition 3.2 of \cite{D},
for any $x\in\fA$ and $\epsilon>0$,
there are finitely many unitaries $z_1,\dots, z_n\in\mathfrak{A}$ such that $\|\frac{1}{n}\sum_{i=1}^{n}z_{i}xz_{i}^{*}-\phi(x)\|<\epsilon$.
Moreover,
since there is a Haar unitary $u$ in $A$ and $B\neq\IC$,
as in the proofs of Lemma 3.1 and Proposition 3.2 of \cite{D}, 
we can take unitaries $z_1,\dots, z_n$ in $(A, \phi_{A})*(B, \phi_{B})(\subset\fA)$.
This implies that the inclusion $(A, \phi_{A})*(B, \phi_{B})\subset\fA$ satisfies Condition ($*$).
\end{exa}

Let $(A, G, \alpha)$ be a \Cs-dynamical system satisfying the assumptions in Section \ref{sec decomp}.
There is a family $\{\rho_\sigma\}_{\sigma\in\widehat{G}}$ of endomorphisms of $A^G$ defined in Section \ref{sec decomp}.
As in Example \ref{exa ab},
if every non-trivial $\rho_{\sigma}$ is properly outer,
then the inclusion $A^{G}\subset A$ satisfies Condition ($*$).
Conversely,
if we assume Condition ($*$),
the following holds for endomorphisms $\{\rho_\sigma\}_{\sigma\in\widehat{G}}$.

\begin{lem}\label{outer1}
Under the above assumptions,
if the inclusion $A^{G}\subset A$ satisfies Condition ($*$),
then for any elements $b, a_1, a_2,\dots, a_{l}\in A^{G}$ with $b\geq0$,
any irreducible representations $\sigma, \sigma_{1}, \sigma_{2},\dots, \sigma_{m}\in\widehat{G}$ with $\sigma\perp\sigma_j$,
and any $\epsilon>0$,
there are finitely many elements $h_1,\dots, h_{n}\in A^{G}$ satisfying 
\[
\|\sum_{i}^{n}h_{i}^{*}h_{i}\|\leq1,
\;\sum_{i}^{n}h_{i}^{*}bh_{i}\approx_{\epsilon}b
\text{ and } \sum_{i}^{n}\rho_{\sigma_j}(h_{i}^{*})a_{k}\rho_{\sigma}(h_{i})\approx_{\epsilon}0
\]
for every $j, k$.
\end{lem}
\begin{proof}
First,
we assume that $\sigma=1_{\widehat{G}}$,
(i.e.,
$\rho_{\sigma}=\id_{A^{G}}$).
Take isometries $v_{j}\in\cH_{\sigma_j}\subset M(A)$ for all $j$.
By (4) of Lemma \ref{easy},
there are finitely many elements $h_1,\dots, h_{m}\in A^{G}$ such that 
\[
\|\sum_{i}h_{i}^{*}h_{i}\|\leq1,
\;\sum_{i}h_{i}^{*}bh_{i}\approx_{\epsilon}b,
\text{ and }\sum_{i}h_{i}^{*}v_{j}^{*}a_{k}h_{i}\approx_{\frac{\epsilon}{\text{dim}(\sigma_j)}}E(\sum_{i}h_{i}^{*}v_{j}^{*}a_{k}h_{i})=0
\]
hold for all $j$ and $k$.
Since we have 
\[
\sum_{i}\rho_{\sigma_j}(h_{i}^{*})a_{k}h_{i}=\text{dim}(\sigma_j)E(v_{j}v_{j}^{*}\sum_{i}\rho_{\sigma_j}(h_{i}^{*})a_{k}h_{i})=\text{dim}(\sigma_j)E(v_{j}\sum_{i}h_{i}^{*}v_{j}^{*}a_{k}h_{i}),
\]
the statement holds.
Moreover,
by Remark \ref{rmk isom},
we can show the same statement for $\sigma_{1}, \sigma_{2},\dots, \sigma_{m}\in \mathrm{\mathrm{Rep}}_{\mathrm{f}}(G)$ with $\sigma_{j}\perp 1_{\widehat{G}}$ and $\sigma=1_{\widehat{G}}$.

For general $\sigma\in\widehat{G}$,
suppose $E_{\sigma}$ is the minimal conditional expectation from $A^{G}$ onto $\rho_{\sigma}(A^{G})$ with index $d$ and $R_{\sigma}\in(\id_{A^G}, \overline{\rho}_{\sigma}\circ\rho_{\sigma})$ is the isometry in Lemma \ref{conjugate}.
As in the proof of Lemma 4.4 in \cite{I1},
we have $E_{\sigma}(x)=\rho_{\sigma}(R_{\sigma}^{*}\overline{\rho}_{\sigma}(x)R_{\sigma})$ for every $x\in A^G$.
Since we have $\overline{\sigma}\otimes\sigma_{j}\perp1_{\widehat{G}}$ and $\rho_{\overline{\sigma}\otimes\sigma_{j}}\sim\overline{\rho}_{\sigma}\circ\rho_{\sigma_j}$ for all $j$, 
there are $h_1,\dots, h_{n}\in A^G$ satisfying
\[
\|\sum_{i}h_{i}^{*}h_{i}\|\leq1,
\;\sum_{i}h_{i}^{*}bh_{i}\approx_{\epsilon}b,\mathrm{\ and\ } 
\sum_{i}\overline{\rho}_{\sigma}\circ\rho_{\sigma_j}(h_{i}^{*})\overline{\rho}_{\sigma}(a_{k})R_{\sigma}h_{i}\approx_{\epsilon}0
\]
for every $j, k$.
We get the statement by the following inequalities:
\begin{align*}
&\big(\sum_{i}\rho_{\sigma_j}(h_{i}^{*})a_{k}\rho_{\sigma}(h_{i})\big)^{*}\big(\sum_{i}\rho_{\sigma_j}(h_{i}^{*})a_{k}\rho_{\sigma}(h_{i})\big)\\
\leq&dE_{\sigma}\Big(\big(\sum_{i}\rho_{\sigma_j}(h_{i}^{*})a_{k}\rho_{\sigma}(h_{i})\big)^{*}\big(\sum_{i}\rho_{\sigma_j}(h_{i}^{*})a_{k}\rho_{\sigma}(h_{i})\big)\Big)\\
= &d\sum_{i, i^{\prime}}\rho_{\sigma}(h_{i}^{*})E_{\sigma}\big(a_{k}^{*}\rho_{\sigma_j}(h_{i}h_{i^{\prime}}^{*})a_{k}\big)\rho_{\sigma}(h_{ i^{\prime}})\\
= &d\sum_{i, i^{\prime}}\rho_{\sigma}(h_{i}^{*})\rho_{\sigma}\Big(R_{\sigma}^{*}\overline{\rho}_{\sigma}\big(a_{k}^{*}\rho_{\sigma_j}(h_{i}h_{i^{\prime}}^{*})a_{k}\big)R_{\sigma}\Big)\rho_{\sigma}(h_{i^{\prime}})\\
= &d\rho_{\sigma}\big(\sum_{i}\overline{\rho}_{\sigma}\circ\rho_{\sigma_j}(h_{i}^{*})\overline{\rho}_{\sigma}(a_{k})R_{\sigma}h_{i}\big)^{*}\rho_{\sigma}\big(\sum_{i}\overline{\rho}_{\sigma}\circ\rho_{\sigma_j}(h_{i}^{*})\overline{\rho}_{\sigma}(a_{k})R_{\sigma}h_{i}\big)<d\epsilon^{2}.
\end{align*}
\end{proof}

To show Theorem \ref{main},
we use the following lemma.

\begin{lem}\label{bimod}
Let $D$ be a simple \Cs-algebra and $\rho\colon D\rightarrow D$ be an irreducible endomorphism with a finite index. 
If a norm closed subspace $\cE\subset D$ satisfies $\rho(D)\cE D\subset\cE$,
then either $\cE=D$ or $\cE=0$.
\end{lem}

\begin{proof}
First,
we assume $\overline{\mathrm{span}}\;\cE\cE^{*}\cap\rho(D)\neq0$.
Since $D\cong\rho(D)$ is simple and the \Cs-subalgebra $\overline{\mathrm{span}}\;\cE\cE^{*}$ of $D$ is closed under multiplications of elements in $\rho(D)$,
we have $\rho(D)\subset\overline{\mathrm{span}}\;\cE\cE^{*}\subset D$.
By Lemma 2.6 of \cite{I1},
there are approximate units $(d_{\nu})_{\nu}$ of $D$ which are contained in $\rho(D)\subset\overline{\mathrm{span}}\;\cE\cE^{*}$.
Since $\cE$ is norm closed and $d_{\nu}x\in\cE$ hold for all $\nu$ and all $x\in D$,
we get $\cE=D$.

Secondly,
suppose $\overline{\mathrm{span}}\;\cE\cE^{*}\cap\rho(D)=0$ and define $B:=\overline{\mathrm{span}}\;\cE\cE^{*}+\rho(D)\subset D$.
The restriction $E_{\rho}|_{B}\colon B\rightarrow\rho(D)$ of the minimal conditional expectation $E_{\rho}$ from $D$ onto $\rho(D)$ is of finite index.
Then,
$\overline{\mathrm{span}}\;\cE\cE^{*}$ is a finite direct sum of simple \Cs-algebras by Theorem 3.4 of \cite{I1}.
Since $\rho$ is irreducible,
we have $M(B)\cap B^{\prime}\subset M(D)\cap\rho(D)^{\prime}=\IC$. 
Thus,
$B$ is simple.
Since $\overline{\mathrm{span}}\;\cE\cE^{*}$ is an ideal of $B$ with $\overline{\mathrm{span}}\;\cE\cE^{*}\cap\rho(D)=0$,
we have $\cE=0$.
\end{proof}

We get the following theorem to give examples of compact group actions for which the Galois correspondence holds.

\begin{thm}\label{main}
Let $\alpha\colon G\acts A$ be a faithful action of a compact second countable group $G$ on a $\sigma$-unital \Cs-algebra $A$.
Suppose $A^G$ is simple and the inclusion $A^{G}\subset A$ satisfies Condition ($*$).
Then,
for every intermediate \Cs-subalgebra $D$ between $A^G$ and $A$,
there exists a unique closed subgroup $H$ of $G$ such that $D=A^H$.

In this case,
the stabilized action $\alpha\otimes {\rm id}_{\IK}\colon G\acts A\otimes\IK$ satisfies the assumptions in Section \ref{sec decomp}.
Under the notations $\{\cH_{\sigma}\subset M(A\otimes\IK)\}_{\sigma\in\widehat{G}}$ and $\eta\colon A\otimes\IK\rightarrow \cE_{E}$ used in Section \ref{sec decomp},
let $B$ be an intermediate \Cs-subalgebra between $(A\otimes\IK)^{G}$ and $A\otimes\IK$.
Put $\cL_{\sigma}:=M(B)\cap\cH_{\sigma}$ for every $\sigma\in\widehat{G}$.
Then,
we have 
\[
\overline{\eta(B)}=\overline{\rm span}\{\xi\in\eta(\cL_{\sigma}^{*}(A\otimes\IK)^{G})\mid \sigma\in\widehat{G}\}.
\]
\end{thm}

\begin{proof}
First,
we show the second half.
For simplicity of presentation,
$(A\otimes\IK,(A\otimes\IK)^{G},\alpha\otimes{\rm id}_{\IK})$ are denoted by $(A_{s},A^{G}_{s},\alpha^{s})$.
By Lemma \ref{easy} (2) and (5),
the inclusion $A_{s}^{G}\subset A_{s}$ satisfies Condition ($*$) and is \Cs-irreducible.
Then,
we get $M(A_{s})\cap (A_{s}^{G})^{\prime}=\IC$ by Remark \ref{rmk irrd}.
This implies that $\alpha^{s}$ satisfies the assumptions in Section \ref{sec decomp}.
For an intermediate subalgebra $B$ in the statement,
the inclusion $\overline{\mathrm{span}}\{\xi\in\eta(\cL_{\sigma}^{*}A_{s}^G)\mid\sigma\in\widehat{G}\}\subset \overline{\eta(B)}$ is trivial.
It suffices to show the converse.
Let $\{u_{\sigma, j}\}_{j=1}^{\text{dim}(\sigma)}$ be orthonormal basis of $\cH_{\sigma}$ such that $\cL_{\sigma}=\mathrm{span}\{u_{\sigma, j}\}_{j=m_{\sigma}}^{\text{dim}(\sigma)}$.
We show that if $j\lneq m_{\sigma}$,
then $E(u_{\sigma, j}B)=0$.
Suppose there is an index $j_{0}\lneq m_{\sigma}$ with $E(u_{\sigma, j_0}B)\neq0$.
We may assume $j_0=1$ without loss of generality.
By Lemma \ref{bimod},
we get $\overline{E(u_{\sigma, 1}B)}=A_{s}^G$.
Let $\overline{u}_{\sigma, i}:=\sqrt{\text{dim}(\sigma)}u_{\sigma, i}^{*}\overline{R}_{\sigma}$.
Since the map $\cH_{\sigma}\ni v\mapsto\sqrt{\text{dim}(\sigma)}v^{*}\overline{R}_{\sigma}\in\cH_{\overline{\sigma}}$ is an antiunitary operator,
we have 
\[\cL_{\overline{\sigma}}=\mathrm{span}\{\overline{u}_{\sigma, i}\}_{i=m_{\sigma}}^{\text{dim}(\sigma)}
\] and \[\overline{E(\overline{u}_{\sigma, 1}^{*}B)}=\overline{R}_{\sigma}^{*}\overline{E(u_{\sigma, 1}B)}=A_{s}^G.
\]
We can take $x\in B$ such that $E(\overline{u}_{\sigma, 1}^{*}x)\geq0$ and $\|E(\overline{u}_{\sigma, 1}^{*}x)\|=1$ hold.
Set $b:=E(\overline{u}_{\sigma, 1}^{*}x)$.
By (4) of Lemma \ref{easy},
for any $\epsilon>0$,
there exist $h_{1}, h_{2},\dots, h_{n}\in A_{s}^{G}$ such that 
\[
\|\sum_{i}h_{i}^{*}h_{i}\|\leq1,\;
\sum_{i}h_{i}^{*}bh_{i}\approx_{\epsilon}b,
\text{ and }\sum_{i}h_{i}^{*}\overline{u}_{\sigma, j}^{*}xh_{i}\approx_{\epsilon}\sum_{i}E(h_{i}^{*}\overline{u}_{\sigma, j}^{*}xh_{i})
\]
hold for $j=1,\dots, \text{dim}(\sigma)$.
Let $a_{j}:=\sum_{i}E(h_{i}^{*}\overline{u}_{\sigma, j}^{*}xh_{i})$ for all $j$.
We have $\|a_j\|\leq\|x\|$.
By Lemma \ref{simple1},
we can take a sequence $\{s_{k}\}_{k=1}^{\infty}$ of $A_{s}^{G}$ and complex numbers $\mu_{1},\dots, \mu_{\text{dim}(\sigma)}\in\IC$ with $\mu_{1}=\|b\|=1$ and $|\mu_{j}|\leq 4\|x\|$ such that 
\[
\sum_{k}s_{k}^{*}s_{k}=1
\text{ and }\sum_{k}s_{k}^{*}a_{j}s_{k}=\mu_{j}1_{M(A_{s}^G)}
\]
hold for all $j$,
where the sums $\sum_{k}s_{k}^{*}s_{k}$ and $\sum_{k}s_{k}^{*}a_{j}s_{k}$ converge in the strict topology.
Hence,
we get 
\begin{equation*}
\sum_{i, k}\rho_{\overline{\sigma}}(s_{k}^{*}h_{i}^{*})xh_{i}s_{k}=\sum_{j}\overline{u}_{\sigma, j}\sum_{i, k}s_{k}^{*}h_{i}^{*}\overline{u}_{\sigma, j}^{*}xh_{i}s_{k}\approx_{2\text{dim}(\sigma)\epsilon}\sum_{j}\mu_{j}\overline{u}_{\sigma, j}=\overline{u}_{\sigma, 1}+\sum_{j=2}^{\text{dim}(\sigma)}\mu_{j}\overline{u}_{\sigma, j}.
\end{equation*}
By Lemma 6.3 of \cite{OP},
the above sums $\sum_{i, k}\rho_{\overline{\sigma}}(s_{k}^{*}h_{i}^{*})xh_{i}s_{k}$ and $\sum_{i, k}s_{k}^{*}h_{i}^{*}\overline{u}_{\sigma, j}^{*}xh_{i}s_{k}$ converge in the strict topology.
Since the left-hand side is contained in $M(B)$ and the right-hand side is contained in the compact set $\{\overline{u}_{\sigma, 1}+\sum_{j=2}^{\text{dim}(\sigma)}d_{j}\overline{u}_{\sigma, j}\mid d_{j}\in\IC,\;|d_{j}|\leq4\|x\|\}$,
we get 
\[
M(B)\cap\{\overline{u}_{\sigma, 1}+\sum_{j=2}^{\text{dim}(\sigma)}d_{j}\overline{u}_{\sigma, j}\mid d_{j}\in\IC\}\neq\emptyset.
\]
This contradicts the assumption that $\overline{u}_{\sigma, 1}$ is orthogonal to $\cL_{\overline{\sigma}}$.
We get $E(u_{\sigma, j}B)=0$ for all $j\lneq m_{\sigma}$.
Since we have $\sum_{\sigma\in\widehat{G}}\text{dim}(\sigma)\sum_{j=1}^{\text{dim}(\sigma)}u_{\sigma, j}^{*}eu_{\sigma, j}=1$ in $\cL(\cE_{E})$,
for every $x\in B$,
\[
\eta(x)=\sum_{\sigma\in\widehat{G}}\text{dim}(\sigma)\sum_{j=m_{\sigma}}^{\text{dim}(\sigma)}\eta(u_{\sigma, j}^{*})E(u_{\sigma, j}x)
\] hold.
We get the second half of the theorem. 

To show the first half of the statement,
consider an intermediate subalgebra $D$ between $A^G$ and $A$.
Let $B$ be $D\otimes\IK$ and put $\cL_{\sigma}:=M(B)\cap\cH_{\sigma}$ as above.
By Lemma 3.16 of \cite{ILP},
there is a unique closed subgroup $H$ of $G$ such that $\cL_{\sigma}=\cH_{\sigma}^{H}$ for every $\sigma\in\widehat{G}$.
Let $E_{H}$ be the natural conditional expectation from $A_{s}$ onto $A_{s}^H$,
then $E_H$ induces an orthogonal projection from $\cH_\sigma$ onto $\cL_\sigma$.
Since every element $x\in A_{s}$ can be approximated by elements in the linear span of $\cup_{\sigma\in\widehat{G}}\cH_{\sigma}^{*}A_{s}^G$,
we have 
\[
A_{s}^{H}=E_{H}(\overline{\mathrm{span}}\{x\in\cH_{\sigma}^{*}A_{s}^G\mid\sigma\in\widehat{G}\})=\overline{\mathrm{span}}\{x\in\cL_{\sigma}^{*}A_{s}^G\mid\sigma\in\widehat{G}\}\subset B.
\]
The converse inclusion $B\subset A_{s}^H$ follows from 
\[
\eta(B)\subset\overline{\mathrm{span}}\{\xi\in\eta(\cL_{\sigma}^{*})A_{s}^G\mid\sigma\in\widehat{G}\}=(\cE_{E})^H.
\]
Therefore,
we get $D\otimes\IK=B=A_{s}^H=A^{H}\otimes\IK$.
This implies $D=A^H$.
\end{proof}

\begin{cor}\label{cor ab}
Let $\alpha\colon G\acts A$ be a quasi-product action of a compact second countable group $G$ on a separable simple \Cs-algebra $A$.
If the fixed point algebra $A^G$ is simple,
then the inclusion $A^{G}\subset A$ is \Cs-irreducible,
and the map
\[
\Phi\colon\{H\mid H\leq G\}\ni H\mapsto A^{H}\in\{D\mid A^{G}\subset D\subset A\}
\] 
is a bijection from the set of all closed subgroups of $G$ onto the set of all intermediate \Cs-subalgebras between $A^G$ and $A$.
\end{cor}
\begin{proof}
The statement follows from Theorem \ref{main} and Example \ref{exa ab}.
\end{proof}
When a compact second countable group $G$ is abelian or profinite,
as noted in the second half of Example \ref{exa ab},
the Galois correspondence also holds for a faithful action $\alpha\colon G\acts A$ on a $\sigma$-unital simple \Cs-algebra such that $A^{G}$ is simple and $M(A)\cap (A^{G})^{\prime}=\IC$.
In particular,
when $G$ is finite and $\alpha$ is outer,
the assumptions are satisfied (see also Corollary 6.6 of \cite{I1}). 

The Galois correspondence for compact abelian group actions is also discussed in \cite{Pel}.
In Theorem 16 of \cite{Pel},
it is proved that
there is a bijective correspondence between the set of closed subgroups and the set of intermediate \Cs-subalgebras, which are closed under the group action.
Let $G$ be a compact second countable abelian group,
and $A$ be a $\sigma$-unital \Cs-algebra.
If an action $G\acts A$ satisfies $M(A)\cap (A^{G})^{\prime}=\IC$ and $A^{G}$ is simple,
then the assumptions in Theorem 16 of \cite{Pel} are satisfied.
In this case,
every intermediate \Cs-subalgebra is automatically closed under the action of $G$.

\begin{cor}\label{cor isa}
Let $G$ be a compact second countable group, and $\alpha\colon G\acts A$ be an isometrically shift-absorbing action on a separable simple \Cs-algebra $A$.
Then,
the inclusion $A^{G}\subset A$ is \Cs-irreducible,
and the map
\[
\Phi\colon\{H\mid H\leq G\}\ni H\mapsto A^{H}\in\{D\mid A^{G}\subset D\subset A\}
\] 
is a bijection from the set of all closed subgroups of $G$ to the set of all intermediate \Cs-subalgebras between $A^G$ and $A$.
\end{cor}
\begin{proof}
The statement follows from Theorem \ref{main} and Example \ref{exa isa1}.
\end{proof}

\begin{cor}\label{cor free}
 Let $G$ be a compact second countable group,
$\alpha\colon G\acts A$ and $\beta\colon G\acts B$ be actions on unital \Cs-algebras with $A^{G},B^{G}\neq\IC$.
Suppose $\phi_{A}$ (resp. $\phi_{B}$) is a $G$-invariant state of $A$ (resp. $B$) such that the restrictions $\phi_{A}|_{A^{G}}\colon A^{G}\rightarrow \IC$ (resp. $\phi_{B}|_{B^{G}}$) induces a faithful GNS representation.
We also suppose that there is a Haar unitary $u$ in the centralizer of $\phi_{A}|_{A^{G}}$. 
The reduced free product \Cs-algebra $(A, \phi_{A})*(B, \phi_{B})$ is denoted by $(\fA, \phi)$ and the free product action $G\acts \fA$ is denoted by $\alpha*\beta$.
If $\alpha*\beta$ is faithful,
then the inclusion $\fA^{G}\subset \fA$ is \Cs-irreducible,
and the map 
\[
\Phi\colon\{H\mid H\leq G\}\ni H\mapsto \fA^{H}\in\{D\mid \fA^{G}\subset D\subset \fA\}
\] 
is a bijection from the set of all closed subgroups of $G$ to the set of all intermediate \Cs-subalgebras between $\fA^G$ and $\fA$.
\end{cor}

\begin{proof}
The statement follows from Theorem \ref{main} and Example \ref{exa free}.
\end{proof}

\begin{exa}
Let $G$ be a compact second countable group,
and $L\colon G\acts C(G)$ be the left translation action.
The free product action $(L\otimes\mathrm{id}_{C(\IT)})*\mathrm{id}_{C(\IT)}\colon G\acts C(G\times\IT)*C(\IT)$ satisfies the assumptions of Corollary \ref{cor free},
where $C(G\times\IT)*C(\IT)$ is the reduced free product with respect to the Haar measures.
\end{exa}

At the end of this section,
we discuss isometrically shift-absorbing actions of compact groups on Kirchberg algebras.
When the \Cs-algebra $A$ is a Kirchberg algebra, and $G$ is abelian or profinite,
isometrically shift-absorbing actions can be characterised as follows.
The idea of the following proof is given by Yuhei Suzuki.
\begin{prop}\label{prop profinite}
Let $\alpha\colon G\acts A$ be an action of a second countable profinite group on a Kirchberg algebra.
The following are equivalent.
\begin{itemize}
    \item[(1)] $\alpha$ is isometrically shift-absorbing. 
    \item[(2)] $\alpha$ is a faithful action such that $A^G$ is purely infinite simple and $M(A)\cap(A^G)^{\prime}=\IC$ holds.
\end{itemize}
\end{prop}
\begin{proof}
The implication $(1)\Rightarrow (2)$ follows from Example \ref{exa isa1} and Remark \ref{rmk_pis}.

We assume $(2)$ holds.
Take a decreasing sequence $\{K_{n}\}_n$ of normal clopen subgroups of $G$ with $\cap_{n}K_{n}=\{1_G\}$.
We claim that the induced action $\alpha_{n}\colon G/K_{n}\acts A^{K_n}$ is isometrically shift-absorbing for every $n$.
Fix $n\in\IN$.
There is a finite index conditional expectation from $A^{K_n}$ onto $A^G$.
By Theorem 3.9 of \cite{I1},
$A^{K_n}$ is purely infinite simple.
Since there is a natural conditional expectation from $A$ onto $A^{K_n}$,
$A^{K_n}$ is also a Kirchberg algebra.
Hence,
it suffices to show that $\alpha_n$ is outer by Theorem 3.15 of \cite{GS}.
As noted after Corollary \ref{cor ab},
the Galois correspondence for $\alpha$ holds.
Then,
$\alpha_n$ is faithful.
Hence, 
we get the claim by assumption $M(A^{K_n})\cap(A^{G})^{\prime}\subset M(A)\cap(A^{G})^{\prime}=\IC$.

To show that $\alpha$ is isometrically shift-absorbing,
it suffices to constract a family $\{S_{\pi,i, j}\}_{\pi\in\widehat{G}, 1\leq i, j\leq\mathrm{dim}(\pi)}$ of isometries in $A_{\alpha,\infty}\cap A^{\prime}$ such that
\[
aS_{\sigma,k, l}^{*}S_{\pi,i, j}=\delta_{(\sigma,k, l),(\pi,i, j)}a
\]
and
\[
a\alpha_{\infty, g}(S_{\pi,i, j})=a\sum_{k=1}^{\mathrm{dim}(\pi)}\langle\pi(g)\xi_{i},\xi_{k}\rangle S_{\pi,k, j}
\]
hold for every $g\in G$, $a\in A$ and $\pi,\sigma\in\widehat{G}$,
where $\xi_{1},\dots,\xi_{\mathrm{dim}(\pi)}$ are orthogonal basis of a representation space of $\pi$.
Since the action $\alpha_{n}\colon G/K_{n}\acts A^{K_n}$ is isometrically shift-absorbing,
we get a family $\{S_{\pi,i, j}^{(n)}\}_{\pi\in\widehat{G/K_n}, 1\leq i, j\leq\mathrm{dim}(\pi)}$ of $A_{\alpha,\infty}\cap (A^{K_n})^{\prime}$ satisfying
\[
aS_{\sigma,k, l}^{(n)*}S_{\pi,i, j}^{(n)}=\delta_{(\sigma,k, l),(\pi,i, j)}a
\]
and
\[
a\alpha_{\infty, g}(S_{\pi,i, j}^{(n)})=a\sum_{k=1}^{\mathrm{dim}(\pi)}\langle\pi(g)\xi_{i},\xi_{k}\rangle S_{\pi,k, j}^{(n)}
\]
for every $g\in G$, $a\in A^{K_n}$ and $\pi,\sigma\in\widehat{G/K_n}$.
Since we have $\widehat{G}=\cup_{n}\widehat{G/K_n}$,
$A=\overline{\cup_{n}A^{K_n}}$ and $A$ is separable,
by a standard reindexation trick and the diagonal argument,
we can construct the family $\{S_{\pi, i, j}\}_{\pi\in\widehat{G}, 1\leq i, j\leq\mathrm{dim}(\pi)}$ with required properties.
\end{proof}

We get a similar result for actions of compact abelian groups.
The idea of the following proof is given by Masaki Izumi.
\begin{prop}\label{prop ab}
Let $\alpha\colon G\acts A$ be a stable action of a compact second countable abelian group on a Kirchberg algebra.
The following are equivalent.
\begin{itemize}
    \item[(1)] $\alpha$ is isometrically shift-absorbing. 
    \item[(2)] $\alpha$ is a faithful action such that $A^G$ is purely infinite simple and $M(A)\cap(A^G)^{\prime}=\IC$ holds.
\end{itemize}  
\end{prop}
\begin{proof}
The implication $(1)\Rightarrow(2)$ follows from Example \ref{exa isa1} and Remark $\ref{rmk_pis}$.

To show the converse,
let $\gamma^{\otimes\infty}\colon G\acts \cO_{\infty}$ be the action in Example \ref{isa exa}.
By Corollary 3.12 and Proposition 3.8 of \cite{GS},
we have $(A,\alpha)\sim_{\KK^G}(A\otimes\cO_{\infty}, \alpha\otimes\gamma^{\otimes\infty})$ and $\alpha\otimes\gamma^{\otimes\infty}$ is isometrically shift-absorbing.
Hence,
both $\alpha$ and $\alpha\otimes\gamma^{\otimes\infty}$ satisfy the assumptions of (2).
By Baaj-Skandalis duality
(see \cite[Theorem 6.19]{BS}), 
\[
(A\rtimes G,\widehat{\alpha})\sim_{\KK^{\widehat{G}}}((A\otimes\cO_{\infty})\rtimes G,\widehat{\alpha\otimes\gamma^{\otimes\infty}})
\]hold.
Since $\alpha$ and $\alpha\otimes\gamma^{\otimes\infty}$ are stable,
Proposition A implies that $A\rtimes G\cong A^G$ and $(A\otimes\cO_{\infty})\rtimes G\cong(A\otimes\cO_{\infty})^G$ are stable Kirchgerg algebras.
Since $M(A\rtimes G)\cap A^{\prime}\subset M(A\rtimes G)\cap (A^{G})^{\prime}=l^{\infty}(\widehat{G})$,
the dual action $\widehat{\alpha}$ is outer.
Similarly,
$\widehat{\alpha\otimes\gamma^{\otimes\infty}}$ is also outer.
By Theorem 6.2 of \cite{GS},
$\widehat{\alpha}$ is cocycle conjugate to $\widehat{\alpha\otimes\gamma^{\otimes\infty}}$.
Hence,
$\alpha$,
$\widehat{\widehat{\alpha}}$,
$\alpha\otimes\gamma^{\otimes\infty}$
and $\widehat{\widehat{\alpha\otimes\gamma^{\otimes\infty}}}$
are cocycle conjugate.
This implies that $\alpha$ is isometrically shift-absorbing.
\end{proof}
Thanks to the above propositions,
when $\alpha\colon G\acts A$ be an action on a separable nuclear \Cs-algebra,
by taking a tensor product of the Cunttz algebra $\cO_{\infty}$,
Corollary \ref{cor ab} follows from Corollary \ref{cor isa}.
Moreover,
the following corollary gives examples of isometrically shift-absorbing actions.
\begin{cor}[See Proposition 7.4 of \cite{Ka}]\label{quasi free}
 Let $G$ be a compact second countable abelian group,
 $\{\omega_{i}\}_{i=1}^{\infty}$ be a sequence of $\widehat{G}$,
 and $\alpha\colon G\acts \cO_{\infty}$ be a quasi-free action induced by the unitary representation $\oplus_{i=1}^{\infty}\omega_{i}\colon G\acts l^{2}(\IN)$.
 Then,
 the following are equivalent.
 \begin{itemize}
     \item[(1)] The action $\alpha\otimes\mathrm{id}_{\IK}\colon G\acts \cO_{\infty}\otimes\IK$ is isometrically shift-absorbing.
     \item[(2)] The dual group $\widehat{G}$ is generated by $\{\omega_{i}\}_{i}$ as a semi-group.
     \item[(3)] The crossed product $\cO_{\infty}\rtimes G$ is purely infinite simple.
 \end{itemize}
\end{cor}
\begin{proof}
The equivalence of $(2)$ and $(3)$ follows from Proposition 7.4 of \cite{Ka}.

The implication $(1)\Rightarrow(3)$ follows from Proposition \ref{prop ab} and Proposition A.

For $(3)\Rightarrow (1)$,
it suffices to show that $\cO_{\infty}\cap(\cO_{\infty}^{G})^{\prime}=\IC$.
Let $\{S_{i}\}_{i=1}^{\infty}$ be generating isometries of $\cO_{\infty}$ satisfying $S_{j}^{*}S_{i}=\delta_{i,j}$ and $\alpha_{g}(S_i)=\omega_{i}(g)S_{i}$ for every $i, j$ and every $g\in G$.
Since $\{\omega_{i}\}_{i}$ generates $\widehat{G}$ as a semi-group,
for every $j$,
there are finite sequences $\{i_{1},\dots,i_{n}\}\subset\IN$  and $\{m_{1},\dots,m_{n}\}\subset\IN$ such that
$\omega_{j}^{-1}=\prod_{k=1}^{n}\omega_{i_k}^{m_k}$.
We define isometries $T_{j}:=S_{j}S_{i_1}^{m_1}\dots S_{i_n}^{m_n}$ of $\cO_{\infty}^{G}$ for all $j$.
Let $\cO_{\infty}\subset B(\cH)$ be a faithful representation and $\{\phi_{n}\colon\cO_{\infty}\rightarrow B(\cH)\}_{n}$ be a sequence of unital completely positive maps defined by
\[
\phi_{n}(x)=\frac{1}{n}\sum_{j=1}^{n}T_{j}^{*}xT_{j}
\]
for all $n$ and $x\in\cO_{\infty}$.
By the constructions of $\{T_{j}\}_j$,
the point $\sigma$-weak cluster point $\phi\colon \cO_{\infty}\rightarrow B(\cH)$ of $\{\phi_{n}\}_{n=1}^{\infty}$ is a state.
For every $x\in\cO_{\infty}\cap(\cO_{\infty}^{G})^{\prime}$,
we have $x=\phi(x)\in\IC$.
Therefore,
we get the statement. 
\end{proof}
\subsection{More general results}
In a very recent article \cite{I2} by Izumi,
remarkable characterizations for quasi-product actions \cite[Theorem 1.1]{I2} and isometrically shift-absorbing actions \cite[Theorem 1.2]{I2} of compact groups on simple \Cs-algebras are proved.
In \cite{I2},
he proved generalized results of Proposition \ref{prop profinite}, Proposition \ref{prop ab}, and Corollary \ref{quasi free} for every compact second countable group that is not necessarily abelian or profinite
(see Theorem 1.2 and Proposition 7.2 of \cite{I2}).

Using \cite[Theorem 1.1]{I2},
when a \Cs-algebra $A$ is separable,
Theorem \ref{main} is improved as follows.
\begin{thm}\label{main2}
    Let $\alpha\colon G\acts A$ be a faithful action of a compact second countable group $G$ on a separable \Cs-algebra $A$.
    If the fixed point algebra $A^{G}$ is simple and the relative commutant $M(A)\cap (A^{G})^{\prime}$ is trivial,
    then the inclusion $A^{G}\subset A$ is \Cs-irreducible,
and the map
\[
\Phi\colon\{H\mid H\leq G\}\ni H\mapsto A^{H}\in\{D\mid A^{G}\subset D\subset A\}
\] 
is a bijection from the set of all closed subgroups of $G$ to the set of all intermediate \Cs-subalgebras between $A^G$ and $A$.
\end{thm}
\begin{proof}
    Theorem 1.1 of \cite{I2} implies that $\alpha$ is a quasi-product action.
    Then,
    we get the statement by Example \ref{exa ab} and Theorem \ref{main}.
\end{proof}

\section*{Appendix}
In this section,
we discuss the simplicity of crossed product \Cs-algebras by minimal actions of compact groups.
We show the following.
\begin{aprop}\label{prop simple}
 Let $\alpha$ be a faithful action of a compact group $G$ on a $\sigma$-unital \Cs-algebra $A$.
 If $A^G$ is simple and $M(A)\cap (A^{G})^{\prime}=\IC$,
 then the crossed product $A\rtimes G$ is simple.
\end{aprop}
It is well-known that if an action of a compact group $G$ on a factor $M$ is minimal,
then the crossed product $M\overline{\rtimes}G$ is a factor.
To prove Proposition A,
we can use the same approach as in the case of factors. 
We use the following notations.
Suppose $\alpha\colon G\acts A$ is an action of a compact group $G$ on a $\sigma$-unital \Cs-algebra $A$.
For every $\pi\in\widehat{G}$,
let $\chi_{\pi}$ be the character of the representation $\pi$ and $A_{1}(\pi):=\{\int_{G}\chi_{\pi}(g)\alpha_{g}(x)dg\mid x\in A\}$ be the spectral subspace of $A$.
(See \cite{Pel1} for details.)
Define $\mathrm{Sp}(\alpha):=\{\pi\in\widehat{G}\mid A_{1}(\overline{\pi})\neq0\}$.
\begin{alem}[See Lemma I\hspace{-1.2pt}I\hspace{-1.2pt}I.3.4 of \cite{AHKT} and Proposition 3.6 of \cite{Ro}]\label{sp}
Let $\pi$ be an irreducible representation of $G$ with $\mathrm{dim}(\pi)=n$ and $V_{\pi}$ be a representation space.
If $\alpha$ is stable,
$A^G$ is simple,
and $M(A)\cap (A^G)^{\prime}=\IC$,
then the following are equivalent.
\begin{itemize}
    \item[(1)] $A_{1}(\overline{\pi})\neq0$.
    \item[(2)] There is a non-zero element $a\in A\otimes B(V_{\pi})$ such that $\overline{\alpha}_{g}(a)=a(1\otimes\pi(g))$ for all $g\in G$,
    where $\overline{\alpha}:=\alpha\otimes\mathrm{id}$.
    \item[(3)] There is a $n$-dimensional Hilbert space $\cH\subset M(A)$ such that $\alpha_{g}(\cH)=\cH$ for every $g\in G$ and the restriction $\alpha|_{\cH}$ is equivalent to $\pi$.
\end{itemize}
\end{alem}
\begin{proof}
 The implication (3) $\Rightarrow$ (1) follows from the Peter-Weyl theorem.

 To show (1) $\Rightarrow$ (2),
 take an element $x\in A$ with $\int_{G}\chi_{\overline{\pi}}(g)\alpha_{g}(x^{*})dg\neq0$ and an orthonormal basis $\{\xi_{1},\dots, \xi_n\}$ of $V_\pi$.
 Define an element $a:=\sum_{i,j}x_{i,j}\otimes e_{i,j}\in A\otimes B(V_\pi)$ as $x_{i,j}:=\int_{G}\langle\pi(g)\xi_{i},\xi_{j}\rangle \alpha_{g}(x)dg$ for every $i,j$,
 where $\{e_{i,j}\}_{i,j}$ are matrix units of $B(V_\pi)$ which correspond to $\xi_{1},\dots, \xi_n$.
 Since $\sum_{i}x_{i,i}^{*}\neq0$,
 we get $a\neq0$ and $\overline{\alpha}_{g}(a)=a(1\otimes\pi(g))$ by straightforward calculation.
 
 To show (2) $\Rightarrow$ (3),
 under the assumption (2),
 we claim that there is a unitary $u\in M(A)\otimes B(V_{\pi})$ with $\overline{\alpha}_{g}(u)=u(1\otimes\pi(g))$ for every $g\in G$.
 As in the proof of Lemma I\hspace{-1.2pt}I\hspace{-1.2pt}I.3.4 of \cite{AHKT},
 we set $\mathscr{P}:=A\otimes B(V_{\pi})\otimes M_2$ and $\sigma:G\acts \mathscr{P}$ as
 \[
 \sigma_{g}\left(
 \begin{bmatrix}
     x_{11}&x_{12}\\
     x_{21}&x_{22}
 \end{bmatrix}
 \right)
 =\begin{bmatrix}
     \overline{\alpha}_{g}(x_{11})&\pi(g)\overline{\alpha}_{g}(x_{12})\\
     \overline{\alpha}_{g}(x_{21})\pi(g)^{*}&\pi(g)\overline{\alpha}_{g}(x_{22})\pi(g)^{*}
 \end{bmatrix},
\]
where we write $\pi(g)$ instead of $1_{M(A)}\otimes\pi(g)$.
Suppose 
$e_{1}:=\begin{bmatrix}
     1&0\\
     0&0
 \end{bmatrix}$
 and 
 $e_{2}:=\begin{bmatrix}
     0&0\\
     0&1
 \end{bmatrix}$.
 By the same argument as in the proof of Lemma I\hspace{-1.2pt}I\hspace{-1.2pt}I.3.4 of \cite{AHKT},
 it suffices to show that $e_1$ and $e_2$ are Murry-von Neumann equivalent to 1 in $M(\mathscr{P}^G)$.
 Since the corners $e_{1}\mathscr{P}^{G}e_{1}\cong A^G\otimes B(V_{\pi})$ and $e_{2}\mathscr{P}^{G}e_{2}\cong (A\otimes B(V_{\pi}))^{\alpha\otimes\mathrm{ad}\pi(G)}$ are stable,
 if they are full corners,
 then we get the claim by Theorem \ref{Brown}.
 Let $\tau$ be the tracial state of $B(V_\pi)$.
 The restriction $\mathrm{id}_{A}\otimes\tau|_{(A\otimes B(V_{\pi}))^{\alpha\otimes\mathrm{ad}\pi(G)}}$ of $\mathrm{id}_{A}\otimes\tau$ is a finite index conditional expectation from $(A\otimes B(V_{\pi}))^{\alpha\otimes\mathrm{ad}\pi(G)}$ onto $A^{G}\otimes 1$.
 Since we have
 \[
 M\big((A\otimes B(V_{\pi}))^{\alpha\otimes\mathrm{ad}\pi(G)}\big)\cap(A^{G}\otimes1)^{\prime}
 =(M(A)\otimes B(V_{\pi}))^{\alpha\otimes\mathrm{ad}\pi(G)}\cap(A^{G}\otimes1)^{\prime}=\IC,
 \]
 $(A\otimes B(V_{\pi}))^{\alpha\otimes\mathrm{ad}\pi(G)}$ is simple by Theorem 3.3 of \cite{I1}.
 By assumption (2),
 we get $e_{1}\mathscr{P}^{G}e_{2}\neq0$.
 Since both $e_{1}\mathscr{P}^{G}e_{1}$ and $e_{2}\mathscr{P}^{G}e_{2}$ are simple,
 they are full corners of $\mathscr{P}^{G}$.
 Hence,
 we get the claim.
 Since $A^G$ is stable,
 there are isometries $S_{1},\dots,S_{n}\in M(A^G)$ with $\sum_{i=1}^{n}S_{i}S_{i}^{*}=1$.
 In the same way as in the proof of Proposition 3.6 of \cite{Ro},
 we can take isometries $v_{1},\dots,v_{n}\in M(A)$ such that 
 $v_{j}^{*}\alpha_{g}(v_{i})=\langle\pi(g)\xi_{i},\xi_{j}\rangle$ for all $i, j$ and $g\in G$.
 We get the Hilbert space $\cH$ in (3) as $\cH:=\mathrm{span}\{v_{1},\dots,v_{n}\}$. 
\end{proof}
\begin{proof}[Proof for Proposition A]
By replacing $(A, \alpha)$ with $(A\otimes\IK, \alpha\otimes\mathrm{id}_\IK)$,
we may assume $\alpha$ is stable.
First,
we show $\mathrm{Sp}(\alpha)=\widehat{G}$.
By using the same strategy as in Appendix C of \cite{AHKT},
it suffices to show that the closed subgroup $K:=\{g\in G\mid \pi(g)=1\mathrm{\ for\ all\ }\pi\in\mathrm{Sp}(\alpha)\}$ of $G$ is a trivial group $\{1_G\}$.
Take $k\in K$.
For every $\pi\in\widehat{G}$ and every $x\in A_{1}(\pi)$,
we have $\alpha_{k}(x)=x$.
Set the projection $q_{\pi}:=\int_{G}\chi_{\pi}(g)\lambda_{g}dg\in\rg(G)\subset M(A\rtimes G)$ for every $\pi\in\widehat{G}$,
where $q_{1_{\widehat{G}}}=p_{1_{\widehat{G}}}$
(see Lemma \ref{lem dual}).
Since the sum $\sum_{\pi\in\widehat{G}}q_{\pi}$ converges to 1 in the strict topology and $\hat{E}(q_{\pi}ap_{1_{\widehat{G}}})\in A_{1}(\pi)$ for all $a\in A$ and $\pi\in\widehat{G}$,
we get 
\begin{align*}
   ap_{1_{\widehat{G}}}
   =&\sum_{\pi\in\widehat{G}}q_{\pi}ap_{1_{\widehat{G}}}\\
   =&\sum_{\pi\in\widehat{G}}\hat{E}(q_{\pi}ap_{1_{\widehat{G}}})p_{1_{\widehat{G}}}\\
   =&\sum_{\pi\in\widehat{G}}\alpha_{k}(\hat{E}(q_{\pi}ap_{1_{\widehat{G}}}))p_{1_{\widehat{G}}}\\
   =&\sum_{\pi\in\widehat{G}}\hat{E}(\lambda_{k}q_{\pi}ap_{1_{\widehat{G}}})p_{1_{\widehat{G}}}\\
   =&\sum_{\pi\in\widehat{G}}\hat{E}(q_{\pi}\alpha_{k}(a)p_{1_{\widehat{G}}})p_{1_{\widehat{G}}}
   =\alpha_{k}(a)p_{1_{\widehat{G}}}
\end{align*}
for all $a\in A$.
Since $\alpha$ is faithful,
we get $k=1_G$.
Hence,
$\mathrm{Sp}(\alpha)=\widehat{G}$ holds.
Thanks to Corollary 3.7 of \cite{Pel1},
the crossed product $A\rtimes G$ is simple if and only if $(A\otimes B(V_{\pi}))^{\alpha\otimes\mathrm{ad}\pi(G)}$ is simple and $A_{1}(\pi)\neq0$ for all $\pi\in\widehat{G}$.
As in the proof of Lemma B,
every $(A\otimes B(V_{\pi}))^{\alpha\otimes\mathrm{ad}\pi(G)}$ is simple.
Therefore,
we get the statement.
\end{proof}

\subsection*{Acknowledgements}
The author is extremely grateful to Yasuyuki Kawahigashi,
who is her supervisor,
for his invaluable support.
She would like to thank David Kerr for supporting her stay in M\"{u}nster during this work.
She would also like to thank Masaki Izumi for many suggestions on various aspects of this paper.
She is grateful to G\'{a}bor Szab\'{o}
and Reiji Tomatsu for their helpful comments.
She is also grateful to Yuhei Suzuki for his valuable comment about Proposition \ref{prop profinite},
Yuki Arano for faithful discussions,
Fuyuta Komura for telling her references \cite{Pel} and \cite{R},
and 
Keigo Yokobori for sharing his master's thesis about relevant results.
She also thank the reviewers for their helpful suggestions.
This work was supported by JSPS KAKENHI Grant Number JP23KJ0560 and the WINGS-FMSP program at the University of Tokyo.
This work was partially supported by JST CREST program JPMJCR18T6.

\end{document}